\documentclass[12pt, a4paper, reqno]{amsart}

\usepackage{amsfonts}
\usepackage{amssymb,amsthm, amsmath}
\usepackage{fancyvrb}
\usepackage{graphicx}
\usepackage[vmargin=2cm,hmargin=2cm]{geometry}
\usepackage[utf8]{inputenc}
\usepackage{xcolor}
\usepackage{pgf,tikz}
 \usetikzlibrary{arrows}
\usepackage{url}
\usepackage{enumitem}

\newtheorem{theorem}{Theorem}
\newtheorem{definition}[theorem]{Definition}
\newtheorem{proposition}[theorem]{Proposition}
\newtheorem{corollary}[theorem]{Corollary}
\newtheorem{lemma}[theorem]{Lemma}

\theoremstyle{remark}

\newtheorem{example}[theorem]{Example}

\newtheorem{remark}[theorem]{Remark}

\title[Binomial ideals and congruences on $\mathbb{N}^n$]{Binomial ideals and congruences on $\mathbb{N}^n$\\ \mbox{}\\
\begin{small}
  Dedicated to Professor Antonio Campillo\\ on the occasion of his 65th birthday.
\end{small} 
}

\author{{Laura Felicia}~Matusevich}
\thanks{The first author was partially supported by NSF grant DMS-1500832}
\address{Mathematics Departament\newline \indent  Texas A\&M University \newline \indent College Station, TX 77843 (USA)}
\email{laura@math.tamu.edu}

\author{Ignacio~Ojeda}
\thanks{The second author was partially supported by the project MTM2015-65764-C3-1, National Plan I+D+I, and by Junta de Extremadura (FEDER funds) - FQM-024}
\address{Departamento de Matem\'{a}ticas\newline \indent Universidad de Extremadura \newline \indent  E-06071 Badajoz (SPAIN)}
\email{ojedamc@unex.es}

\begin{document}

%%%%%%%%%%%%%%%%%%%%%%%%%%%%%%%%%%%%%%%%%%%%%%%%%%%%%%%%%%%%%%%%%%%%

\begin{abstract}
A \emph{congruence} on $\mathbb{N}^n$  is an equivalence relation on
$\mathbb{N}^n$ that is compatible with the additive structure. If
$\Bbbk$ is a field, and $I$ is a \emph{binomial ideal}
in $\Bbbk[X_1,\dots,X_n]$ (that is, an ideal generated by polynomials
with at most two terms), then $I$ induces a congruence on 
$\mathbb{N}^n$ by declaring $\mathbf{u}$ and $\mathbf{v}$ to be
equivalent if there is a linear combination with nonzero coefficients
of $\mathbf{X}^{\mathbf{u}}$ and $\mathbf{X}^{\mathbf{v}}$ that
belongs to $I$. While every congruence on $\mathbb{N}^n$ arises this
way, this is not a one-to-one correspondence, as many binomial
ideals may induce the same congruence. Nevertheless, the link between
a binomial ideal and its corresponding congruence is strong, and one may
think of congruences as the underlying combinatorial structures of
binomial ideals. In the current literature, the theories of binomial ideals and congruences on $\mathbb{N}^n$ are developed separately. The aim of
this survey paper is to provide a detailed parallel exposition, that provides algebraic intuition for the combinatorial analysis of congruences. For the elaboration of this survey paper, we followed mainly \cite{KM} with an eye on \cite{ES96} and \cite{OjPie}.
\end{abstract}


\maketitle

%%%%%%%%%%%%%%%%%%%%%%%%%%%%%%%%%%%%%%%%%%%%%%%%%%%%%%%%%%%%%%%%%%%%
%%%%%%%%%%%%%%%%%%%%%%%%%%%%%%%%%%%%%%%%%%%%%%%%%%%%%%%%%%%%%%%%%%%%
%\section*{Introduction}

%%%%%%%%%%%%%%%%%%%%%%%%%%%%%%%%%%%%%%%%%%%%%%%%%%%%%%%%%%%%%%%%%%%%
%%%%%%%%%%%%%%%%%%%%%%%%%%%%%%%%%%%%%%%%%%%%%%%%%%%%%%%%%%%%%%%%%%%%
\section{Preliminaries}\label{Sect1}

In this section we introduce our main objects of study: binomial
ideals and monoid congruences, and recall some basic 
results.

Throughout this article, $\Bbbk[\mathbf{X}] := \Bbbk[X_1, \ldots,
X_n]$ is the commutative polynomial
ring in $n$ variables over a field
$\Bbbk$. In what follows we write $\mathbf{X}^\mathbf{u}$ for
$X_1^{u_1} X_2^{u_2} \cdots X_n^{u_n}$, where $\mathbf{u} = (u_1, u_2,
\ldots, u_n) \in \mathbb{N}^n$,  where here and henceforth,
$\mathbb{N}$ denotes the set of nonnegative integers. 

%\medskip
%%%%%%%%%%%%%%%%%%%%%%%%%%%%%%%%%%%%%%%%%%%%%%%%%%%%%%%%%%%%%%%%%%%%

\subsection{Binomial ideals}\mbox{}\par

In this section we begin our study of binomial ideals. First of all, we recall that a \textbf{binomial} in $\Bbbk[\mathbf{X}]$ is a polynomial with at
most two terms, say $\lambda \mathbf{X}^\mathbf{u} + \mu
\mathbf{X}^\mathbf{v}$, where $\lambda, \mu \in \Bbbk$ and
$\mathbf{u}, \mathbf{v} \in \mathbb{N}^n$. We emphasize that, according to this definition, monomials are binomials.

\begin{definition}
\label{def:binomialIdeal}
A \textbf{binomial ideal} of $\Bbbk[\mathbf{X}]$ is an ideal of
$\Bbbk[\mathbf{X}]$ generated by binomials. 
\end{definition}

Throughout this article, we assume that the base field $\Bbbk$ is
algebraically closed. The reason for this is that some desirable
results are not valid over an arbitrary field. These include the 
characterization of binomial prime ideals (Theorem~\ref{Th
  CarPrimos}), and the fact that associated primes of binomial ideals
are binomial (see, e.g. Proposition~\ref{Prop Celular-CP}). This failure can be seen even
in one variable: the ideal $\langle X^2+1 \rangle \subset
\mathbb{R}[X]$ is prime, but does not conform to the description in 
Theorem~\ref{Th CarPrimos}; the ideal $\langle X^3-1 \rangle \subset
\mathbb{R}[X]$ has the associated prime $\langle X^2+X+1 \rangle$,
which is not binomial. It is also worth noting that the characteristic
of $\Bbbk$ plays a role when studying binomial ideals, as can be seen
by the different behaviors presented by $\langle X^p-1 \rangle \subset
\Bbbk[X]$ depending on whether the characteristic of $\Bbbk$ is $p$.

The following result is an invaluable tool when studying binomial ideals.

\begin{proposition}\label{Prop1.1}
Let $I \subset \Bbbk[\mathbf{X}]$ be an ideal. The following are equivalent:
\begin{enumerate}
\item
$I$ is a binomial ideal. \label{Prop1.1:itema}
\item
The reduced Gr\"obner basis of $I$ with respect to any monomial order
on $ \Bbbk[\mathbf{X}]$ consists of binomials.\label{Prop1.1:itemb} 
\item
A universal Gr\"obner basis of $I$ consists of
binomials.\label{Prop1.1:itemc} 
\end{enumerate}
\end{proposition}

\begin{proof}
If $I$ has a binomial generating set, the $S$-polynomials produced by
a step in the Buchberger algorithm are necessarily binomials. 
\end{proof} 

Since the Buchberger algorithm for computing Gr\"obner bases respects the binomial condition, Gr\"obner techniques are particularly effective when
working with these objects. In particular, it can be shown that some important
ideal theoretic operations preserve binomiality. For instance, it is easy to show that eliminating variables from binomial ideals
results in binomial ideals.

\begin{corollary}
\label{coro:elimIsBinomial}
Let $I$ be a binomial ideal of $\Bbbk[\mathbf{X}]$. The
elimination ideal $I \cap \Bbbk[X_i \mid i \in \sigma]$ is a binomial ideal
for every nonempty subset $\sigma \subset \{1,\dots,n\}$.
\end{corollary}

\begin{proof}
The intersection is generated by a subset of the reduced Gr\"obner
basis of $I$ with respect to a suitable lexicographic order.  
%Recall that if $\mathcal{G}$ is a Gr\"obner basis with respect to the
%lexicographical term order $\prec$ on $\Bbbk[\mathbf{X}]$ with $X_1
%\prec \ldots \prec X_n$, then a system of generators of $I \cap
%\Bbbk[X_1, \ldots, X_r]$ is $\mathcal{G} \cap \Bbbk[X_1, \ldots,
%X_r]$. 
\end{proof}

\begin{example}
Let $\varphi : \Bbbk[X,Y,Z] \to \Bbbk[T]$ be the $\Bbbk-$algebra
morphism such that \[X \to T^3, Y \to T^4\ \text{ and } Z \to T^5.\] It is known
that $\ker(\varphi) = \langle X - T^3, Y - T^4, Z - T^5 \rangle \cap
\Bbbk[X,Y,Z]$. As a consequence of
Corollary~\ref{coro:elimIsBinomial}, $\ker(\varphi)$ is a binomial
ideal. In fact, $\ker(\varphi)$ is the ideal generated by $\{Y^2-XZ, X^2Y-Z^2, X^3-YZ \},$ as can be checked by executing the following code in Macaulay2 (\cite{M2}): 
\begin{verbatim}
      R = QQ[X,Y,Z,T]
      I = ideal(X-T^3,Y-T^4,Z-T^5)
      eliminate(T,I)
\end{verbatim}
\end{example}

Taking ideal quotients is a fundamental operation in commutative
algebra. We can now show that some ideal quotients of binomial ideals
are binomial.

\begin{corollary}
\label{coro:colonIdealIsBinomial}
If $I$ is a binomial ideal of $\Bbbk[\mathbf{X}]$, and
$\mathbf{X}^{\mathbf{u}}$ is a monomial, then
$(I:\mathbf{X}^{\mathbf{u}})$ is a binomial ideal. 
\end{corollary}

\begin{proof}
Recall that if $\{f_1,\dots,f_\ell\}$ is a system of generators for
$I\cap\langle \mathbf{X}^{\mathbf{u}} \rangle$, then
$\{f_1/\mathbf{X}^{\mathbf{u}},$ $\dots,$ $f_\ell/\mathbf{X}^{\mathbf{u}}\}$
is a system of generators for $(I:\mathbf{X}^{\mathbf{u}})$. 
Thus, the binomiality of $(I:\mathbf{X}^{\mathbf{u}})$ follows if we show that 
$I\cap\langle \mathbf{X}^\mathbf{u} \rangle$ is binomial.

Introducing an auxiliary variable $T$, we have that \[I\cap\langle
\mathbf{X}^{\mathbf{u}} \rangle = (TI + (1-T)
\langle\mathbf{X}^{\mathbf{u}}\rangle) \cap \Bbbk[\mathbf{X}].\] Since
$TI + (1-T) \langle \mathbf{X}^{\mathbf{u}} \rangle$ is a binomial
ideal, Corollary~\ref{coro:elimIsBinomial} implies that $I\cap\langle
\mathbf{X}^{\mathbf{u}} \rangle$ is also binomial, as we wanted. 
\end{proof}

We remark that the ideal quotient of a binomial ideal by a binomial is
not necessarily binomial, and neither is the ideal quotient of a
binomial ideal by a monomial ideal. When taking colon with a single
binomial, the above proof breaks because the product of two binomials
is not a binomial in general; indeed, \[\big(\langle X^3 - 1 \rangle : \langle X-1 \rangle \big) = \langle X^2+X+1 \rangle \subset \Bbbk[X]. \] In the case of taking ideal quotient by a monomial ideal, say $J = \langle \mathbf{X}^{\mathbf{u}_1},\dots
\mathbf{X}^{\mathbf{u}_r}\rangle$, instead of a single monomial, what
makes the argument invalid is that the ideal $(I : J)$ is equal to $\cap_{i=1}^r (I:\langle \mathbf{X}^{\mathbf{u}_i}
\rangle)$, and the intersection of binomial ideals is not necessarily
binomial, as the following shows: $\langle X - 1 \rangle \cap \langle X - 2 \rangle = \langle X^2-3X+2 \rangle \subset \Bbbk[X]$.

%%%%%%%%%%%%%%%%%%%%%%%%%%%%%%%%%%%%%%%%%%%%
\subsection{Graded algebras}\mbox{}\par
%%%%%%%%%%%%%%%%%%%%%%%%%%%%%%%%%%%%%%%%%%%%

Gradings play a big role when studying binomial ideals. The main
result of this section is that a ring is a quotient of a polynomial
ring by a binomial ideal if and only if it has a special kind of
grading (Theorem~\ref{Prop1.11ES}).

Recall that a $\Bbbk-$algebra of finite type $R$ is \emph{graded} by a
finitely generated commutative monoid $S$ if $R$ is a direct sum  
\[
R = \bigoplus_{\mathbf{a} \in S} R_\mathbf{a}
\]
of $\Bbbk-$vector spaces and the multiplication of $R$ satisfies the
rule $R_\mathbf{a} R_{\mathbf{a}'} = R_{\mathbf{a} + \mathbf{a}'}$. 

\begin{example}
Observe that $\Bbbk[\mathbf{X}] = \bigoplus_{\mathbf{u} \in
  \mathbb{N}^n} {\rm Span}_\Bbbk \{\mathbf{X}^\mathbf{u}\}$. 
\end{example}

\begin{remark}
\label{remark:insteadOfLema1.1}
Let $I$ be \emph{any} ideal of $\Bbbk[\mathbf{X}]$ and let $\pi$ be
the canonical projection of $\Bbbk[\mathbf{X}]$ onto $R :=
\Bbbk[\mathbf{X}]/I$. Let $S$ be the set of all one-dimensional
subspaces ${\rm Span}_\Bbbk \{ \pi(\mathbf{X}^\mathbf{u}) \}$ of $R$;
if the kernel of $\pi$ contains monomials, we adjoin to $S$ the symbol  $\infty$ associated to the
monomials in $\ker(\pi)$. The set $S$ is a commutative monoid
with the operation  
\[
{\rm Span}_\Bbbk \{ \pi(\mathbf{X}^{\mathbf{u}}) \} + {\rm Span}_\Bbbk \{ \pi(\mathbf{X}^{\mathbf{v}}) \} = {\rm Span}_\Bbbk \{ \pi (\mathbf{X}^{\mathbf{u}}\mathbf{X}^{\mathbf{v}}) \} = {\rm Span}_\Bbbk \{ \pi(\mathbf{X}^{\mathbf{u}+\mathbf{v}})\}
\] 
and identity element ${\rm Span}_\Bbbk \{ 1 \} = {\rm Span}_\Bbbk \{
\pi(\mathbf{X}^{\mathbf{0}})\}$. Note that if $\mathbf{X}^{\mathbf{v}}
\in \ker(\pi)$, then for any other monomial $\mathbf{X}^{\mathbf{u}}$,
$\mathbf{X}^{\mathbf{u}}\mathbf{X}^{\mathbf{v}}\in \ker(\pi)$. In
other words, \[{\rm Span}_\Bbbk \{ \pi(\mathbf{X}^{\mathbf{u}}) \} +
\infty = \infty.\]
We point out that the set 
$\{ {\rm Span}_\Bbbk \{\pi(X_1)\}, \dots, {\rm Span}_\Bbbk \{\pi(X_n)\} \}$ generates $S$ as a monoid.
There is a natural $\Bbbk-$vector space surjection
\begin{equation}\label{ecu1.1} 
\bigoplus_{\substack{ {\rm Span}_\Bbbk \{ \pi(\mathbf{X}^\mathbf{u}) \} \in S \\ \mathbf{X}^\mathbf{u} \not\in \ker\pi }} 
{\rm Span}_\Bbbk \{ \pi(\mathbf{X}^\mathbf{u}) \} \rightarrow R.
\end{equation}
We observe that if~\eqref{ecu1.1} is an isomorphism of $\Bbbk-$vector
spaces, then $R$ is \textbf{finely graded} by $S$, meaning that $R$ is
$S-$graded  
and every graded piece has dimension at most $1$.
%\begin{lemma}\label{Lema1.1}
%Let $R$ and $S$ be as above. Then $R$ is $S$-graded if and only if the natural map 
%\begin{equation}\label{ecu1.1} 
%\bigoplus_{\substack{ \langle \pi(\mathbf{X}^\mathbf{u}) \rangle  \in
%S \\ \mathbf{X}^\mathbf{u} \not\in \ker\pi }}  
%\langle \pi(\mathbf{X}^\mathbf{u}) \rangle \rightarrow R
%\end{equation} 
%is an isomorphism of vector spaces.
%\end{lemma}
%
%\begin{proof}
%It is clear that if~\eqref{ecu1.1} is a vector space isomorphism, then $R$ is $S$-graded. For the converse, note that~\eqref{ecu1.1} is always surjective.
%If this map is not injective, there exists $f \in \Bbbk[\mathbf{X}]$ 
%with at least two terms $\mathbf{X}^\mathbf{u}$ and $\mathbf{X}^\mathbf{v}$ of different $S-$degree 
%such that $\pi(f) = 0$ and no proper subset sum of terms of $f$ goes to zero via $\pi$
%\end{proof}
%
%
% Laura's reason for changing things: I am confused by the iff proof,
% it seemed like circular reasoning to me. Fortunately, only the easy
% direction is needed in the Theorem below, so I downgraded this from
% a Lemma to a Remark. Also, it makes me uncomfortable to use the same
% notation for two different things, and $\langle \rangle$ for me is
% reserved for ideals (not vector spaces). 
\end{remark}

The following result provides the first link between binomial ideals and monoids.

\begin{theorem}\label{Prop1.11ES}
%\cite[Proposition 1.11]{ES96}.
%
% Laura's reason for changing: I do not think we need a precise reference here.
%
A $\Bbbk-$algebra $R$ of finite type admits a presentation of the form
$\Bbbk[\mathbf{X}]/I$, where $I$ is a binomial ideal, if and only if
$R$ can be finely graded by a finitely genera\-ted commutative monoid.% with $n$ generators. % in such a way that every homogeneous component of $R$ has dimension at most $1$. 
\end{theorem}

\begin{proof}
%First suppose that $R$ admits a grading of the given type by a
%commutative monoid $S$ generated by $\mathcal{A} = \{\mathbf{a}_1,
%\ldots, \mathbf{a}_n\}$. Let $\varphi_\mathcal{A} : \Bbbk[\mathbf{X}]
%\to R$ a $\Bbbk-$algebra homomorphism given by sending the variable
%$X_i$ to a chosen nonzero element of the one-dimensional space
%$R_{\mathbf{a_i}}$, for $i=1, \ldots, n$. 
%
%{\bf Laura:} Need to show that $\varphi_\mathcal{A}$ is
%surjective. This is not obvious, and indeed not true if the grading
%is not fine. 

First assume that $R$ admits a grading of the given type by a
finitely generated commutative monoid $S$.
Let $f_1,\dots,f_n$ be $\Bbbk-$algebra generators of $R$. Without loss
of generality, we may assume that $f_1,\dots,f_n$ are
homogeneous. Denote $\mathbf{a}_i$ the degree of $f_i$, for
$i=1,\dots,n$. Since $R$ is (finely) graded by the monoid generated by 
$\{ \mathbf{a}_1,\dots, \mathbf{a}_n\}$, we may assume that $S$ is
generated by $\{ \mathbf{a}_1,\dots, \mathbf{a}_n\}$.

Give $\Bbbk[\mathbf{X}]$ an $S-$grading by setting the
degree of $X_i$ to be $\mathbf{a}_i$, and consider 
the surjection  $\Bbbk[\mathbf{X}] \to R$
given by $X_i\mapsto f_i$, which is a graded ring homomorphism. The
kernel of this map is a homogeneous ideal of $\Bbbk[\mathbf{X}]$, and
is therefore generated by homogeneous elements. 
On the other hand, by the fine grading condition, for any two monomials $\mathbf{X}^{\mathbf{u}},
\mathbf{X}^{\mathbf{v}} \in \Bbbk[\mathbf{X}]$ with the same
$S-$degree, neither of which maps to zero in $R$, 
there is a scalar $\lambda \in \Bbbk^*$ such that the
binomial  
$\mathbf{X}^{\mathbf{u}} - \lambda \mathbf{X}^{\mathbf{v}} \in
\Bbbk[\mathbf{X}] $ maps to zero in $R$. Thus, the kernel of the above
surjection is generated by binomials.

%The relations on these generators are generated by homogenous
%relations, that is, by relations that are sums of monomials all with
%the same $S-$degree. But for any two monomials
%$\mathbf{X}^{\mathbf{u}}, \mathbf{X}^{\mathbf{v}} \in
%\Bbbk[\mathbf{X}]$ with the same $S-$degree, there is an scalar
%$\lambda \in \Bbbk^*$ such that the binomial  
%$\mathbf{X}^{\mathbf{u}} - \lambda \mathbf{X}^{\mathbf{v}} \in
%\Bbbk[\mathbf{X}] $ goes to zero in $R$. Thus, the ideal of all such
%binomials generates $I := \ker \varphi_\mathcal{A}$. 

Conversely, by Remark~\ref{remark:insteadOfLema1.1}, it suffices to
show that the map~\eqref{ecu1.1} is injective. We need to prove that
if $\Sigma$ is a nonempty subset of $ \{ {\rm Span}_\Bbbk\{
\pi(\mathbf{X}^\mathbf{u}) \} \in S \mid \mathbf{X}^\mathbf{u} \not\in
\ker\pi \}$, then the image of $\Sigma$ in $R$ is linearly independent.
This follows if we show that if $f = \sum_{i=1}^r \lambda_i \mathbf{X}^{\mathbf{u}_i} \in I$
with $\lambda_1,\dots,\lambda_r \in \Bbbk^*$ and
$\mathbf{X}^{\mathbf{u}_i} \notin I$ for all $1 \leq i \leq r$, then
there exist $1 \leq j \leq r$ and $\lambda \in \Bbbk^*$ such that
$\mathbf{X}^{\mathbf{u}_1} - \lambda \mathbf{X}^{\mathbf{u}_j} \in I$
(in other words, $\pi(\mathbf{X}^{\mathbf{u}_1}) = \pi(\mathbf{X}^{\mathbf{u}_j})$).
To see this, note that since $I$ is a binomial ideal, it has a
$\Bbbk-$vector space basis consisting of binomials, and therefore we
can write $f = \sum_{i=1}^\ell \mu_i B_i$, where $\mu_1,\dots,\mu_\ell
\in \Bbbk^*$ and each $B_i$ is a binomial in $I$ with two terms,
neither of which is in $I$ (the latter by the assumption on $f$). The
monomial $\mathbf{X}^{\mathbf{u}_1}$ must appear in at least one of
the binomials $B_1,\dots, B_\ell$, say $B_{i_1}$. Of course, the
second monomial appearing in $B_{i_1}$ has the same image under $\pi$ as
$\mathbf{X}^{\mathbf{u}_1}$. 
If this second monomial in $B_{i_1}$ is one of the
$\mathbf{X}^{\mathbf{u}_2},\dots,\mathbf{X}^{\mathbf{u}_r}$, we are
done. Otherwise, the second term of $B_{i_1}$ must appear in another 
of the binomials $B_i$, say $B_{i_2}$. Note that both monomials in
$B_{i_2}$ have the same image under $\pi$ as $\mathbf{X}^{\mathbf{u}_1}$. 
If the second monomial of
$B_{i_2}$ is one of the
$\mathbf{X}^{\mathbf{u}_2},\dots,\mathbf{X}^{\mathbf{u}_r}$, again, we
are done. Otherwise, continue in the same manner. Since we only have
finitely many binomials to consider, this process must stop, and
produce a monomial $\mathbf{X}^{\mathbf{u}_j}$ such that 
$\pi(\mathbf{X}^{\mathbf{u}_1}) = \pi(\mathbf{X}^{\mathbf{u}_j})$.
%
%%% The fact that normal forms are terms is not proven; I included a
%%% self-contained proof.
%
%If we choose a term order
%on $\Bbbk[\mathbf{X}]$, then by Proposition~\ref{Prop1.1}\ref{Prop1.1:itemb} the
%normal form of a monomial modulo $I$ is a term. This implies that if
%$\langle \pi(\mathbf{X}^\mathbf{u}) \rangle$ is contained in linear
%span of  $\langle \pi(\mathbf{X}^{\mathbf{u}_1}) \rangle, \ldots,
%\langle \pi(\mathbf{X}^{\mathbf{u}_t}) \rangle$ modulo $I$, then
%$\langle \pi( \mathbf{X}^{\mathbf{u}_1}) \rangle$ is contained in one
%of the $\langle \pi(\mathbf{X}^{\mathbf{u}_i}) \rangle$ modulo $I$.
\end{proof}

%\medskip
%%%%%%%%%%%%%%%%%%%%%%%%%%%%%%%%%%%%%%%%%%%%%%%%%%%%%%%%%%%%%%%%%%%%%%%
\section{Congruences on monoids and binomial ideals} %\mbox{}\par
%%%%%%%%%%%%%%%%%%%%%%%%%%%%%%%%%%%%%%%%%%%%%%%%%%%%%%%%%%%%%%%%%%%%%%%

We now start our study of monoid congruences, and their relationship
to binomial ideals. We show how binomial ideals induce congruences,
and how any congruence can arise this way. We also address the
question of when two different binomial ideals give rise to the same congruence.

\begin{definition}
\label{def:congruence}
Let $S$ be a commutative monoid. A \textbf{congruence} $\sim$ on $S$
is an equivalence relation on $S$ which is additively closed: $\mathbf{a}
\sim \mathbf{b} \Rightarrow \mathbf{a} + \mathbf{c} \sim
\mathbf{b}+\mathbf{c}$ for $\mathbf{a},\mathbf{b}$ and $\mathbf{c} \in
S$.  
\end{definition}

The following result, which follows directly from the definition,
gives a first indication that congruences on commutative monoids
are analogous to ideals in commutative rings.

\begin{proposition}
If $\sim$ is a congruence on a commutative monoid $S$, then $S/\!\sim$
is a commutative monoid. 
\qed
\end{proposition}

%\begin{proof}
%\textcolor{red}{Exercise.}
%\end{proof}

Let $\phi : S \to S'$ be a monoid morphism.The \textbf{kernel of
  $\phi$} is defined as 
\[
\ker \phi := \big\{ (\mathbf{a}, \mathbf{b}) \in S \times S \mid
\phi(\mathbf{a}) = \phi(\mathbf{b}) \big\}.
\] 
Note that if $\phi$ is a monoid morphism, the relation on $S$
determined by $\ker \phi \subset S\times S$ is actually a
congruence. Moreover, every congruence on $S$ arises in this way:
if $\sim$ is a congruence on $S$, then $\sim$ can be recovered as the
congruence induced by the kernel of the natural surjection
$S \to S/\!\sim$.

%Clearly, congruences on $S$ are exactly the kernels of monoid morphism $\phi : S \to S'$.

We write $\mathrm{cong}(S) \subset \mathcal{P}(S \times S)$ for the
set of congruences on $S$ ordered by inclusion. (Here $\mathcal{P}$
indicates the power set.)
We say that $S$ is \textbf{Noetherian} if every nonempty subset of
$\mathrm{cong}(S)$  has a maximal element (equivalently,
$\mathrm{cong}(S)$  satisfies the ascending chain condition). 
The following is an important result in monoid theory.

\begin{theorem}
\label{thm:NoetherianMonoidCondition}
A commutative monoid $S$ is Noetherian if and only if $S$ is finitely generated.
\end{theorem}

The fact that a Noetherian monoid is finitely generated is the hard part of the proof. It is due to
Budach~\cite{Budach}, and is the main result in Chapter~5 in Gilmer's
book~\cite{Gilmer}, where it appears as Theorem~5.10. Brookfield has
given a short and self contained proof in~\cite{Brookfield}.  
We will just provide a proof of the converse, namely, that finitely generated
monoids are Noetherian  (see \cite[Theorem 7.4]{Gilmer}), after Theorem~\ref{Th 1.1}.

%\begin{proof}
%See \cite[Section 7]{Gilmer}.
%\end{proof}

Set $S$ be a commutative monoid finitely generated by $\mathcal{A} =
\{\mathbf{a}_1, \ldots, \mathbf{a}_n\}$. The monoid
morphism \begin{equation}\label{ecu1} \pi : \mathbb{N}^n
  \longrightarrow S;\ \mathbf{e}_i \longmapsto  \mathbf{a}_i,\ i = 1,
  \ldots, n,\end{equation} where $\mathbf{e}_i$ denotes the element in
$\mathbb{N}^n$ whose $i-$th coordinate is $1$ with all other
coordinates $0$, is surjective and gives a \textbf{presentation} $$S =
\mathbb{N}^n/\! \sim$$ by simply taking $\sim = \ker \pi$. Unless
stated otherwise, we write $[\mathbf{u}]$ for the class of $\mathbf{u}
\in \mathbb{N}^n$ modulo $\sim$. 

\begin{remark}
In what follows, all monoids considered are commutative and finitely generated.
\end{remark}

Given a monoid $S$, the \textbf{semigroup algebra} $\Bbbk[S] :=
\bigoplus_{\mathbf{a} \in S} {\rm Span}_\Bbbk \{ \chi^\mathbf{a}\}$ is the direct sum
with multiplication $\chi^\mathbf{a} \chi^\mathbf{b} =
\chi^{\mathbf{a}+\mathbf{b}}$. (This terminology is in wide use, even
though the algebra $\Bbbk[S]$ would be more
precisely named a ``monoid algebra''.)

\begin{theorem}\label{Th 1.1}
Let $\mathcal{A} =\{\mathbf{a}_1,\dots,\mathbf{a}_n\}$ be a generating
set of a monoid $S$, and consider the
presentation map $\pi:\mathbb{N}^n \to S$  induced by $\mathcal{A}$. We define a map of
semigroup algebras
\begin{equation}\label{ecu2a}
%\label{eqn:monomialMap}
\hat\pi : \Bbbk[\mathbb{N}^n] = \Bbbk[\mathbf{X}]  \to  \Bbbk[S] \ ; \quad
\mathbf{X}^\mathbf{u}  \mapsto  \chi^{\pi(\mathbf{u})}.  
\end{equation}
Let
\begin{equation}\label{ecu2} 
I_{\mathcal{A}}
:=  
%\langle \mathbf{X}^\mathbf{u} - \mathbf{X}^\mathbf{v}\ \mid \
%\mathbf{u} \sim \mathbf{v} \rangle =  
\langle \mathbf{X}^\mathbf{u} - \mathbf{X}^\mathbf{v}\ \mid \
\pi(\mathbf{u}) = \pi(\mathbf{v}) \rangle 
\subseteq  
\Bbbk[\mathbf{X}].
\end{equation}
Then $\ker \hat\pi = I_{\mathcal{A}}$, so that
$\Bbbk[S]  \cong \Bbbk[\mathbf{X}]/I_{\mathcal{A}}$. Moreover,
$I_{\mathcal{A}}$ is spanned as a $\Bbbk-$vector space by 
$\{\mathbf{X}^\mathbf{u} - \mathbf{X}^\mathbf{v} \mid  \pi(\mathbf{u}) = \pi(\mathbf{v})\}$.
\end{theorem}

\begin{proof}
By construction, $I_{\mathcal{A}} \subseteq \ker \hat\pi$. To prove the other inclusion,
give $\Bbbk[\mathbf{X}]$ an $S-$grading by setting $\deg(X_i) =
\pi(\mathbf{e}_i)=\mathbf{a}_i$. Then the map $\hat \pi$ is graded (considering
$\Bbbk[S]$ with its natural $S-$grading), and therefore its
kernel is a homogeneous ideal of $\Bbbk[\mathbf{X}]$.
Note that 
$\mathbf{X}^{\mathbf{u}}$ and $\mathbf{X}^{\mathbf{v}}$ have the same
$S-$degree if and only if $\pi(\mathbf{u})=\pi(\mathbf{v})$.

We observe that $\ker\hat\pi$ contains no monomials, so any polynomial in
$\ker\hat\pi$ has at least two terms. 
Let $f$ be a homogeneous element of $\ker\hat\pi$.  
Then there are $\lambda,\mu \in \Bbbk^*$ and
$\mathbf{u},\mathbf{v} \in \mathbb{N}^n$ such that $f=\lambda
\mathbf{X}^{\mathbf{u}} + \mu \mathbf{X}^{\mathbf{v}} + g$, with $g$ a
homogeneous polynomial with two fewer terms than $f$. Since $f$ is
homogeneous, we have that $\pi(\mathbf{u})=\pi(\mathbf{v})$, and therefore
$\mathbf{X}^{\mathbf{u}}-\mathbf{X}^{\mathbf{v}} \in I_{\mathcal{A}} \subset
\ker \hat\pi$. Then
$f-\lambda(\mathbf{X}^{\mathbf{u}}-\mathbf{X}^{\mathbf{v}})$ is a
homogeneous element of $\ker\hat\pi$, and has fewer terms than
$f$. Continuing in this manner, we conclude that $f \in I_{\mathcal{A}}$. Since
$\ker\hat\pi$ is a homogeneous ideal, we see that $I_{\mathcal{A}}
\supseteq \ker \hat\pi$, and therefore $I_{\mathcal{A}}=\ker \hat\pi$.

For the final statement, we note that any binomial ideal in $\Bbbk[\mathbf{X}]$ is spanned as
a $\Bbbk-$vector space by the set of all of its binomials. Since
$I_{\mathcal{A}}$ contains no monomials and is $S-$graded, any binomial in
$I_{\mathcal{A}}$ is of the form $\mathbf{X}^{\mathbf{u}}-\lambda
\mathbf{X}^{\mathbf{v}}$, where $\lambda \in \Bbbk^*$ and
$\pi(\mathbf{u})=\pi(\mathbf{v})$. But then $\mathbf{X}^{\mathbf{u}}-
\mathbf{X}^{\mathbf{v}} \in I_{\mathcal{A}}$, and again using that
$I_{\mathcal{A}}$ contains no monomials, we see that $\lambda=1$. This
implies that 
$\{\mathbf{X}^\mathbf{u} - \mathbf{X}^\mathbf{v} \mid  \pi(\mathbf{u}) =
\pi(\mathbf{v})\}$
is the set of all binomials of $I_{\mathcal{A}}$, which implies that
it is a $\Bbbk-$spanning set for this ideal.
%
% I changed the proof again, mostly because the notation used in
% the proof below had not been introduced yet. The argument is the same.
%
%
%It suffices to show that each polynomial in $I_\mathcal{A}$ is a
%$\Bbbk-$linear combination of binomials $\mathbf{X}^\mathbf{u} -
%\mathbf{X}^\mathbf{v},\ \mathbf{u}, \mathbf{v} \in \mathbb{N}^n$,
%with $\deg_\mathcal{A}(\mathbf{u}) =
%\deg_\mathcal{A}(\mathbf{v})$. Fix a term order $\prec$ on
%$\Bbbk[\mathbf{X}]$. Suppose that $f \in I_\mathcal{A}$ cannot be
%written as a $\Bbbk-$linear combination of binomials. We choose a
%polynomial $f$ with this property such that the initial term
%$\mathrm{in}_\prec(f) = \mathbf{X}^\mathbf{u}$ is minimal with
%respect to the term order. When expanding $f(\chi-^{\mathbf{a}_1},
%\ldots, \chi^{\mathbf{a}_n})$ we get zero. In particular, the term
%$\chi^{\deg_\mathcal{A}(\mathbf{u})} =
%\varphi_A(\mathbf{X}^\mathbf{u}) $ must cancel during this
%expansion. Hence there exists some monomial $\mathbf{X}^\mathbf{v}
%\prec  \mathbf{X}^\mathbf{u}$ appering in $f$ such that
%$\pi(\mathbf{u}) = \pi(\mathbf{v})$. Also the polynomial $f' = f -
%(\mathbf{X}^\mathbf{u} - \mathbf{X}^\mathbf{v})$ cannot be written as
%a $\Bbbk-$linear combination of binomials in $I_\mathcal{A}$. But
%since $\mathrm{in}_\prec(f') \prec \mathrm{in}_\prec(f)$ this is a
%contradiction. 
\end{proof}

We are now ready to prove that finitely generated monoids are Noetherian.

\begin{proof}[Proof of
  Theorem~\ref{thm:NoetherianMonoidCondition}, reverse implication]
Let $S$ be a finitely generated monoid, and consider a presentation $S
= \mathbb{N}^n/\! \sim$, where $\sim$ is a congruence on
$\mathbb{N}^n$. In this proof, for $\mathbf{u} \in \mathbb{N}^n$, 
we denote by $[\mathbf{u}]$ the equivalence class of $\mathbf{u}$ with
respect to $\sim$.

Let $\approx$ be a congruence on $S$, and let $\simeq$ be the congruence on
$\mathbb{N}^n$ given by setting the equivalence class of $\mathbf{u}
\in \mathbb{N}^n$ with respect to $\simeq$ to be the set 
$\bigcup_{\{\mathbf{v} \in \mathbb{N}^n \mid [\mathbf{u}] \approx [\mathbf{v}]\}} [\mathbf{v}]$. 
Then the congruence $\simeq$ is such that $S/\!\approx\ = \mathbb{N}^n/\!\simeq$.

Now let $\approx_1$ and $\approx_2$ be two congruences on $S$ and consider
the natural surjections $\pi_i : \mathbb{N}^n \to
\mathbb{N}^n/\! \simeq_i$ for $i=1,2$. Then if $\approx_1 \ \subseteq \ \approx_2$
(as subsets of $S\times S$), we have that $I_{\mathcal{A}_1} \subseteq
I_{\mathcal{A}_2}$, where these ideals are defined as in~\eqref{ecu2} by considering the generating sets $\mathcal A_i = \{ \pi_i(\mathbf{e}_j)\ \mid\ j=1, \ldots, n\},\ i = 1, 2,$ respectively. We
conclude that Noetherianity of the monoid $S$ follows from the fact
that $\Bbbk[\mathbf{X}]$ is a Noetherian ring.
\end{proof}

In order to continue to explore the correspondence between congruences and
binomial ideals, we introduce some terminology.

\begin{definition}
A binomial ideal is said to be \textbf{unital} if it is generated by
binomials of the form $\mathbf{X}^\mathbf{u} - \lambda
\mathbf{X}^\mathbf{v}$ with $\lambda$ equal to either $0$ or $1$. A
binomial ideal is said to be \textbf{pure} if does not contain any
monomial. 
\end{definition}

\begin{corollary}\label{Cor 11}
A relation $\sim$ on $\mathbb{N}^n$ is a congruence if and only if
there exists a pure unital ideal $I \subset \Bbbk[\mathbf{X}]$ such
that $\mathbf{u} \sim \mathbf{v} \Longleftrightarrow
\mathbf{X}^\mathbf{u} -  \mathbf{X}^\mathbf{v} \in I$. 
\end{corollary} 

\begin{proof}
If $\sim$ is a congruence on $\mathbb{N}^n$, then
$\mathbb{N}^n/\!\sim$ is a (finitely generated) monoid. Consider the
natural surjection $\pi:\mathbb{N}^n\to\mathbb{N}^n/\! \sim$, and let
$\mathcal{A}=\{\pi(\mathbf{e}_1), \dots, \pi(\mathbf{e}_n)\}$.
Use this information to construct $I_{\mathcal{A}}$ as in~\eqref{ecu2}. By
Theorem~\ref{Th 1.1} and its proof, the ideal $I_{\mathcal{A}}$ satisfies the required
conditions.

For the converse, let $I$ a pure unital ideal of
$\Bbbk[\mathbf{X}]$ such that 
$\mathbf{u} \sim \mathbf{v} \Longleftrightarrow \mathbf{X}^\mathbf{u} -  \mathbf{X}^\mathbf{v}
\in I$. 
Clearly, $\sim$ is reflexive and symmetric. For transitivity,
it suffices to observe that 
$\mathbf{X}^\mathbf{u} - \mathbf{X}^\mathbf{w} = (\mathbf{X}^\mathbf{u} -
\mathbf{X}^\mathbf{v}) + (\mathbf{X}^\mathbf{v} -
\mathbf{X}^\mathbf{w}) \in I$, for every $\mathbf{u}, \mathbf{v}$ and
$\mathbf{w}$ such that $\mathbf{u} \sim \mathbf{v}$ and $\mathbf{v}\sim \mathbf{w}$.
Finally, as $I$ is an ideal, it follows that
$\mathbf{X}^\mathbf{w} (\mathbf{X}^\mathbf{u} -
\mathbf{X}^\mathbf{v})  = \mathbf{X}^{\mathbf{u}+\mathbf{w}} -
\mathbf{X}^{\mathbf{v}+\mathbf{w}} \in I$, for every
$\mathbf{X}^\mathbf{u} - \mathbf{X}^\mathbf{v} \in I$ and
$\mathbf{X}^\mathbf{w} \in \Bbbk[\mathbf{X}]$. We conclude that $\sim$ is a
congruence. 
%
%If $\sim$ is a congruence on $\mathbb{N}^n$, then
%$\mathbb{N}^n/\!\sim$ is (a finitely generated) monoid. Thus, by
%taking $I$ defined as in (\ref{ecu2}), the direct implication follows
%from Theorem \ref{Th 1.1}. 
\end{proof}

We review some examples of pure unital binomial ideals and their associated
congruences. We remark in particular that different binomial ideals
may give rise to the same congruence.

\begin{example}
\label{ejem10}
\mbox{}\par
\begin{enumerate}
\item[(i)] 
The ideal $I = \langle X-Y \rangle \subset \Bbbk[X,Y]$ defines a
congruence $\sim$ on $\mathbb{N}^2$ with $\mathbb{N}^2/\!\sim =
\mathbb{N}$. 
\item%[(ii)] 
The ideal $I = \langle X-Y,Y^2-1 \rangle \subset \Bbbk[X,Y]$ defines a congruence $\sim$
on $\mathbb{N}^2$ such that $\mathbb{N}^2/\!\sim =
\mathbb{Z}/2\mathbb{Z}$. 
\item%[(iii)] 
The ideal $I = \langle X^2-Y^2 \rangle \subset \Bbbk[X,Y]$ defines a
congruence $\sim$ on $\mathbb{N}^2$ such that $\mathbb{N}^2/\!\sim$ is
isomorphic to the submonoid $S$ of $\mathbb{Z} \oplus
\mathbb{Z}/2\mathbb{Z}$ generated by $(1,0)$ and $(1,1)$. 
\item%[(iv)]
\label{item{iv}}
Consider the monoid $S = \{0,a,b\}$  where the sum is defined as follows:
$$
\begin{array}{c|c|c|c|}
+ & 0 & a & b \\
\hline
0 & 0 & a & b \\
\hline
a & a &b & b \\
\hline
b & b & b & b \\
\hline
\end{array}
$$
The ideal $I = \langle X-Y,Y^3-Y^2 \rangle \subset \Bbbk[X,Y]$ determines a congruence $\sim$ on  $\mathbb{N}^2$ such that $S \cong \mathbb{N}^2/\!\sim$.
\end{enumerate}
\end{example}

An arbitrary binomial ideal $J$ of $\Bbbk[\mathbf{X}]$ induces a congruence
$\sim_J$ on $\mathbb{N}^n$ defined as 
\begin{equation}\label{simI}
\mathbf{u} \sim_J \mathbf{v} \Longleftrightarrow \text{there exists }
\lambda \in \Bbbk^* \text{ such that } \mathbf{X}^\mathbf{u}
- \lambda \mathbf{X}^\mathbf{v} \in J.
\end{equation}
Note that this ideal defines the same congruence as the pure unital binomial ideal
\[
I = \langle \mathbf{X}^\mathbf{u} - \mathbf{X}^\mathbf{v} \mid
\text{there exists } \lambda \in \Bbbk^* \text{ such that } 
\mathbf{X}^\mathbf{u} - \lambda \mathbf{X}^\mathbf{v} \in J \rangle.
\] 

\begin{example}
\label{ejem10b}
\mbox{}\par
\begin{enumerate}
\item Let $J = \langle X-Y,Y^2 \rangle \subset \Bbbk[X,Y]$. The congruence $\sim_J$ induced by $J$ on $\mathbb{N}^2$ is exactly the same that one in Example \ref{ejem10} \ref{item{iv}}.
\item The congruence $\sim_{\langle X,Y \rangle}$  on $\mathbb{N}^2$ is the same as the induced by $I=\langle X-Y,X-X^2\rangle$ on $\mathbb{N}^2$. Note that $\langle X,Y \rangle$ is a monomial ideal, while $I$ contains no monomials.
\end{enumerate}
\end{example}

If a binomial ideal $I$ contains monomials, then the exponents of all
monomials in $I$ form a single equivalence class in the congruence
$\sim_I$. This
equivalence class satisfies an absorption property, as in the
definition below.

\begin{definition}
\label{def:nil}
A non-identity element $\infty$ in a monoid $S$ is \textbf{nil} if
$\mathbf{a} + \infty = \infty$, for all $\mathbf{a} \in S$. 
\end{definition}

For example, the ``formal'' element $\infty$ introduced in
Remark~\ref{remark:insteadOfLema1.1} is nil, since it corresponds to
the monomial class. Note that a monoid $S$ can have at most one nil
element: if $\infty, \infty' \in S$ are both nil, then
$\infty+\infty'=\infty'$ because $\infty'$ is nil, and 
$\infty'+\infty=\infty$ because $\infty$ is nil. Since $S$ is
commutative, $\infty=\infty'$.

As we have noted above, if $I$ is a binomial ideal that contains
monomials, then the class of monomial exponents is a nil element for
the congruence $\sim_I$. The converse of this assertion is false: if
$J$ is a binomial ideal containing monomials, then the ideal $I$
produced by Corollary~\ref{Cor 11} for the congruence $\sim_J$ has no
monomials and has a nil element (since $J$ contains monomials, and
therefore $\sim_J$ does). On the other hand, if $\sim$ is a congruence
on $\mathbb{N}^n$ with a nil element $\infty$, then there exists a binomial
ideal $J$ in $\Bbbk[\mathbf{X}]$ that contains monomials, and such
that $\sim = \sim_J$. To see this, let $I$ be the ideal produced by
Corollary~\ref{Cor 11} for $\sim$, and consider 
$J=I + \langle \mathbf{X}^{\mathbf{e}} \mid [\mathbf{e}] = \infty
\rangle$, noting that adding this particular monomial ideal does not
change the underlying congruence. We make this more precise in
Proposition~\ref{Prop 9.5KM}.

%Observe that a binomial ideal $I$ of $\Bbbk[\mathbf{X}]$ contains
%monomials if and only if $\mathbb{N}^n/\!\sim_I$ has a nil $\infty$,
%
% Laura: The above is false, and the proof below is wrong: in the
% grading given by $\mathbf{N}^n/\!\sim$, you may have a graded piece
% corresponding to a nil that does not have dimension 0.
%
%
%
%because $\infty$ is the class of all exponents of moomials in
%$I$. More precisely, under the grading of the quotient algebra
%$\Bbbk[\mathbf{X}]/I$ by $\mathbb{N}^n/\!\sim_I$ the dimension of the
%graded piece of $(\Bbbk[\mathbf{X}]/I)_{[\mathbf{u}]}$ as a
%$\Bbbk-$vector space is either $0$ or $1$, depending on whether $I$
%contains a monomial in the corresponding class. Since the (exponents
%on) monomials in $I$ form a single class, the dimension can only be
%$0$ for at most one $[\mathbf{u}]$, and $[\mathbf{u}]$ must be a nil
%in $\mathbb{N}^n/\!\sim_I$.  
%
%
%
%\medskip

\begin{proposition}\label{Prop 9.5KM}
Let $I \subset \Bbbk[\mathbf{X}]$ be a binomial ideal. If $J$ is a
binomial ideal of $ \Bbbk[\mathbf{X}]$ such that $I \subset J$ and
$\sim_J=\sim_I$, then $\mathbb{N}^n/\!\sim_I$ has a nil $\infty$ and $J = I+\langle \mathbf{X}^\mathbf{e} \mid [\mathbf{e}] = \infty \rangle$.
\end{proposition}

\begin{proof}
As $I \subset J$, there is a binomial $\mathbf{X}^\mathbf{u} -
\lambda \mathbf{X}^\mathbf{v} \in J \setminus I$. Since $\sim_I =
\sim_J,$ necessarily $\mathbf{X}^{\mathbf{u}},  \mathbf{X}^{\mathbf{v}}
\in J$; in particular $\mathbb{N}^n/\!\sim_J = \mathbb{N}^n/\!\sim_I$
has a nil $\infty$. 
We claim that the ideal $J$ is equal to $I + \langle \mathbf{X}^{\mathbf{e}} \mid
[\mathbf{e}]=\infty\rangle$. To see that $J$ contains 
$I + \langle \mathbf{X}^{\mathbf{e}} \mid [\mathbf{e}]=\infty\rangle$, we note
that $I \subset J$. Also, we know that $J$ contains a monomial
$\mathbf{X}^{\mathbf{u}}$, and so $[\mathbf{u}]=\infty$. If $\mathbf{e} \in
\mathbb{N}^n$ is such that $[\mathbf{e}]=\infty = [\mathbf{u}]$, then 
$\mathbf{X}^{\mathbf{u}} - \mu \mathbf{X}^{\mathbf{e}} \in J$ for
some $\mu \in \Bbbk^*$, and since $\mathbf{X}^{\mathbf{u}} \in J$, we
see that $\mathbf{X}^{\mathbf{e}}\in J$. For the reverse inclusion, it
is enough to see that any binomial in $J$ belongs to 
$I + \langle \mathbf{X}^{\mathbf{e}} \mid [\mathbf{e}]=\infty\rangle$.
But as before, if $\mathbf{X}^\mathbf{u} -
\lambda \mathbf{X}^\mathbf{v} \in J \setminus I$, then 
$\mathbf{X}^{\mathbf{u}},  \mathbf{X}^{\mathbf{v}} \in J$, and
therefore $[{\mathbf{u}}]=[{\mathbf{v}}]=\infty$, because a monoid can
have at most one nil element.
%
%
% Laura: This proof seems wrong to me. It does not show that $J$ is the
% only possible ideal with this property.
%
%Thus, by Lemma \ref{Lemma above Prop 9.5KM}, there
%exists a monoid ideal $E$ of $\mathbf{N}^n$ such that $\langle
%\mathbf{X}^\mathbf{e}\ \mid \mathbf{e} \in E \rangle  = \langle
%\mathbf{X}^\mathbf{e} \mid [\mathbf{e}] = \infty \rangle \subseteq J$ and
%we are done. 
\end{proof}

A \textbf{monoid ideal} $E$ of $\mathbb{N}^n$ is a proper subset such that $E
+ \mathbb{N}^n \subseteq E$; Figure~\ref{fig:monoidIdeal} shows a
typical example.

\begin{figure}[h]
\begin{center}
\includegraphics[scale=0.75]{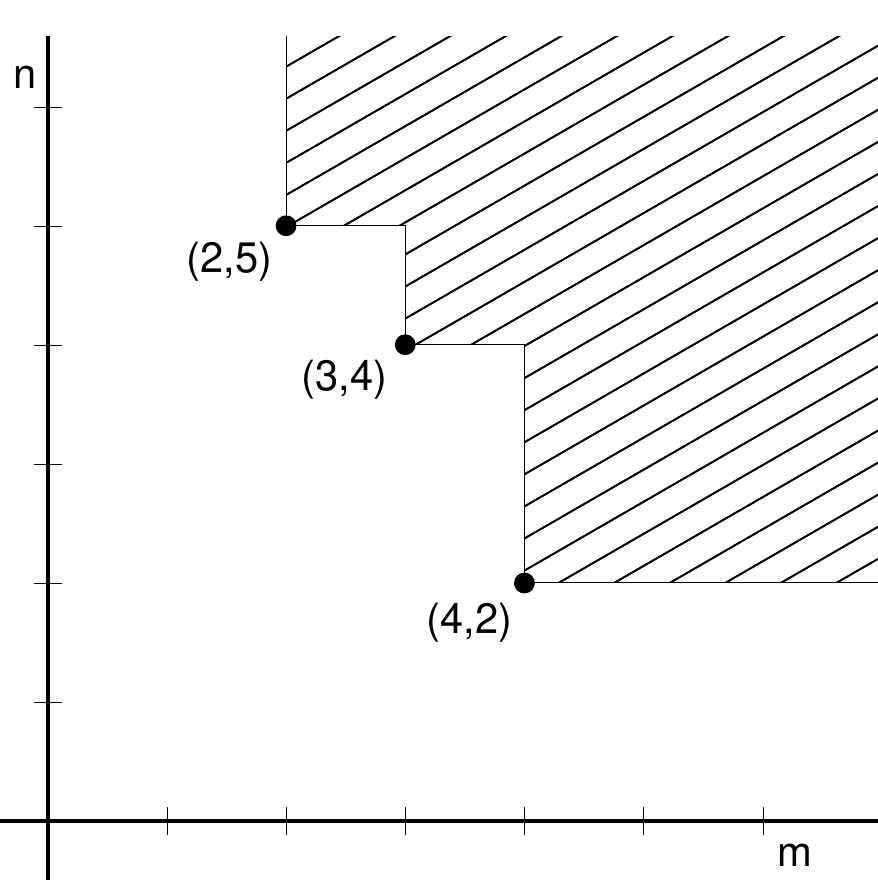}
\end{center}
\caption{The integer points in shaded area form a monoid ideal of $\mathbb{N}^2$.}
\label{fig:monoidIdeal}
\end{figure}

Let $E \subseteq \mathbb{N}^n$ be a
monoid ideal of $\mathbb{N}^n$. The \textbf{Rees congruence} on
$\mathbb{N}^n$ modulo $E$ is the correspondence $\sim$ on
$\mathbb{N}^n$ defined by 
$\mathbf{u} \sim \mathbf{v} \Longleftrightarrow \mathbf{u} =
\mathbf{v}$ 
or both $\mathbf{u}$ and $\mathbf{v} \in E$. Notice that the Rees
congruence on $\mathbb{N}^n$ modulo $E$ is the same as the induced by
$\sim_{M_E}$ with 
$M_E = \langle \mathbf{X}^\mathbf{e}\ \mid \mathbf{e} \in E \rangle$. 

Monoid ideals and nil elements are related as follows.

\begin{lemma}\label{Lemma above Prop 9.5KM}
Let $S$ be a monoid. Then $S$ has a nil element if and only if for any
presentation $\mathbb{N}^n/\!\sim$ of $S$ there exists a monoid ideal
$E$ of $\mathbb{N}^n$ such that $\sim$ contains the Rees congruence on 
$\mathbb{N}^n$ modulo $E$. In this case, $[\mathbf{e}] = \infty,$ for
any $\mathbf{e} \in E$. 
\end{lemma}

\begin{proof}

% Laura: I broke paragraphs to make the proof easier to read.

Let $\mathbb{N}^n/\!\sim$ be a presentation of $S$ given by a
monoid surjection $\pi :  \mathbb{N}^n \to S$. 

For the direct implication, assume that $\infty \in S$ is a nil.
Then $E := \pi^{-1}(\infty)$ is a monoid ideal of $\mathbb{N}^n$.
Indeed, given $\mathbf{u} \in \mathbb{N}^n$ and $\mathbf{e} \in E$ we
have that $$\pi(\mathbf{e} + \mathbf{u}) = \pi(\mathbf{e}) +
\pi(\mathbf{u}) = \infty +  \pi(\mathbf{u}) = \infty,$$ so that
$\mathbf{e} + \mathbf{u} \in E$. 
Note that $E \neq \mathbb{N}^n$ since nil elements are nonzero.
Moreover, by construction, if $\mathbf{e}, \mathbf{e}' \in E,$ then 
$\mathbf{e} \sim \mathbf{e}'$, 
which means that $\sim$ contains the Rees congruence on
$\mathbb{N}^n$ modulo $E$. 

Conversely, let $E$ be a monoid ideal of $\mathbb{N}^n$ such that  
$\sim$ contains the Rees congruence on $\mathbb{N}^n$ modulo $E$. We
claim that the class $[\mathbf{e}]$ (for any $\mathbf{e} \in E$) is a
nil in $S = \mathbb{N}^n/\!\sim$. To see this, 
let $\mathbf{e} \in E$, $\mathbf{u}\in \mathbb{N}^n$. Then
$\pi(\mathbf{e}) + \pi(\mathbf{u}) = \pi(\mathbf{e} + \mathbf{u})$.
Since $E$ is a monoid ideal, $\mathbf{e} + \mathbf{u} \in
E$. This implies that $\mathbf{e} \sim \mathbf{e} + \mathbf{u}$ (or equivalently,
$\pi(\mathbf{e}) = \pi(\mathbf{e} + \mathbf{u})$) because $\sim$
contains the Rees congruence modulo $E$. To complete the
proof of our claim, we need to show that $[\mathbf{e}]$ ($\mathbf{e}
\in E$) is not the zero class. This follows from the fact that $E \neq \mathbb{N}^n$.
\end{proof}

Our next goal is to prove Theorem~\ref{KMThm9.12}, which is a more
precise version of Proposition~\ref{Prop 9.5KM}. With that result in
hand, we will be able to introduce the binomial ideal associated to a
congruence in Definition~\ref{def:binomialIdealAssocCong}.

\begin{definition}
\label{def:augmentationIdeal}
%The \textbf{unital augmentation ideal} in $\Bbbk[\mathbf{X}]$ is the
%maximal ideal $I^1_{\text{aug}} := \langle X_1 - 1, \ldots, X_n - 1
%\rangle.$ More generaly, 
An \textbf{augmentation ideal} for a given
binomial ideal $I \subset \Bbbk[\mathbf{X}]$ is a maximal ideal of
the form 
\[
I_{\mathrm{aug}} := \langle X_i - \lambda_i\ \mid  \lambda_i
\in \Bbbk^*,\ i = 1, \ldots, n \rangle
\] 
such that $I \cap I_{\mathrm{aug}}$ is a binomial ideal. 
\end{definition}

We point out that, given a binomial ideal $I$, an augmentation ideal
for $I$ may or may not exist (see~\cite[Example~9.13]{KM} for a
binomial ideal without an augmentation ideal). The following result
is the $\mathbb{N}^n$-version of~\cite[Theorem 9.12]{KM}.

\begin{theorem}%\cite[Theorem 9.12]{KM}.
\label{KMThm9.12}
If $I_\ell \supset \ldots \supset I_0$ is a chain of distinct binomial ideals of
$\Bbbk[\mathbf{X}]$ inducing the same congruence on $\mathbb{N}^n$,
then $\ell \leq 1$. Moreover, if $\ell = 1$ then $I_0$ is pure and $I_1$ is not: $I_0 = I_1 \cap I_{\mathrm{aug}}$ for an
augmentation ideal for $I_1$. 
\end{theorem}

\begin{proof}
By Proposition~\ref{Prop 9.5KM}, all we need to show is that if
$\ell=1$, then $I_0 = I_1 \cap I_{\mathrm{aug}}$, where
$I_{\mathrm{aug}}$ is an augmentation ideal for $I_1$. Denote by
$\sim$ the congruence induced by $I_0$ (and $I_1$). 

Assume $\ell =1$, so that $I_0$ does not have monomials, and $I_1$
does. In particular, we may select a monomial
$\mathbf{X}^{\mathbf{e}} \in I_1$, and its equivalence class $[\mathbf e]$
with respect to $\sim$ is a nil element, that
we denote $\infty$. For each $1\leq i \leq n$, consider the monomial
$X_i = \mathbf{X}^{\mathbf{e}_i}$. Since $[\mathbf{e}_i] + \infty =
\infty$, there exists $\lambda_i \in \Bbbk^*$ such that 
$X_i \mathbf{X}^{\mathbf{e}} - \lambda_i \mathbf{X}^{\mathbf{e}} \in
I_0$ (because $I_0$ and $I_1$ induce the same congruence). 
We now define $I_{\mathrm{aug}} = \langle X_i-\lambda_i \mid i=1,\dots,n \rangle$,
and claim that $I_1 \cap I_{\mathrm{aug}} = I_0$, which in particular
shows that $I_{\mathrm{aug}}$ is an augmentation ideal for $I_1$.

By construction, $(I_0:\mathbf{X}^{\mathbf{e}}) \supseteq
I_{\mathrm{aug}}$. Note that $(I_0:\mathbf{X}^{\mathbf{e}}) \neq
\langle 1 \rangle$, as $I_0$ contains no monomials. Thus,
since $I_{\mathrm{aug}}$ is maximal, $(I_0:\mathbf{X}^{\mathbf{e}}) =
I_{\mathrm{aug}}$, and we conclude that $I_{\mathrm{aug}}$ contains
$I_0$. This, and $I_1 \supset I_0$, imply that $I_1 \cap
I_{\mathrm{aug}} \supseteq I_0$. 
Moreover, $I_{\mathrm{aug}} \not \supseteq I_1$ because $I_1$
has monomials, while $I_{\mathrm{aug}}$ does not. Consequently 
$I_1 \supsetneq I_1 \cap I_{\mathrm{aug}} \supseteq I_0$. Now the equality $I_1
\cap I_{\mathrm{aug}} = I_0$ will follow from Proposition~\ref{Prop
  9.5KM} if we show that $I_1 \cap I_{\mathrm{aug}}$ is binomial
(since the fact that $I_1$ and $I_0$ induce the same congruence $\sim$
implies that the congruence induced by $I_1 \cap I_{\mathrm{aug}}$ is
also $\sim$). To see that $I_1 \cap I_{\mathrm{aug}}$ is binomial, we
use the argument from~\cite[Corollary~1.5]{ES96}. Introduce an
auxiliary variable $t$, and consider the binomial ideal $J = I_0 +
tI_{\mathrm{aug}} +(1-t) \langle \mathbf{X}^{\mathbf{u}} \mid [\mathbf{u}]= 
[\mathbf{e}] \rangle \subset \Bbbk[\mathbf{X},t]$. Since, by Proposition~\ref{Prop 9.5KM}, 
$I_1= I_0 + \langle \mathbf{X}^{\mathbf{u}} \mid [\mathbf{u}]=\infty =
[\mathbf{e}] \rangle$, we have that $J \cap
\Bbbk[\mathbf{X}] = I_1 \cap I_{\mathrm{aug}}$. Now apply Corollary~\ref{coro:elimIsBinomial}.

%%%%%%%%%%%%%%%%%%%%%%%%%%%%%%
%
%
%The first sentence follows from Proposition \ref{Prop 9.5KM}, as does
%the statament about monomials when $\ell = 1$. It remains to show that
%$I_0 = I_1 \cap I_\mathrm{aug}$ if $\ell = 1$. Set $I = I_0$. Under
%the grading of the quotient algebra $\Bbbk[\mathbf{X}]/I$ by
%$\mathbb{N}^n/\!\sim_I$ the dimension of the grade piece
%$(\Bbbk[\mathbf{X}]/I)_{[\mathbf{u}]}$ as a $\Bbbk-$vector space is
%$1$ for all $\mathbb{N}^n/\!\sim_I$. By Proposition \ref{Prop 9.5KM},
%there is a nil $\infty \in \mathbb{N}^n/\!\sim_I$ which is the class
%of all exponents on monomials in $I_1$. Fix a nonzero element
%$\mathbf{X}^\mathbf{e}$ of degree $\infty$. For each $\mathbf{v} \in
%\mathbb{N}^n$, there exists $\lambda_\mathbf{v} \in \Bbbk^*$ such that
%$\mathbf{X}^\mathbf{v} \mathbf{X}^\mathbf{e} - \lambda_\mathbf{v}
%\mathbf{X}^\mathbf{e} \in I$, because $[\mathbf{v}] + [\mathbf{e}] =
%[\mathbf{v}]  + \infty = \infty =  [\mathbf{e}] $ and for every
%$\mathbf{e}' \sim_I \mathbf{e},$ one has that
%$\mathbf{X}^{\mathbf{e}'} - \lambda_{\mathbf{e}'}
%\mathbf{X}^\mathbf{e} \in I$. Set $I_\mathrm{aug} = \langle
%\mathbf{X}^\mathbf{v} - \lambda_\mathbf{v} \mid \mathbf{v} \in
%\mathbb{N}^n \rangle$. Then $I_\mathrm{aug} \supseteq I$ by
%construction, but $I_\mathrm{aug} \not\supseteq I_1$, since $I_1$
%contains monomials and $I_\mathrm{aug}$ does not. Therefore $I_1
%\supsetneq I_1 \cap I_\mathrm{aug} \supseteq I$, whence $I_1 \cap
%I_\mathrm{aug} = I$, because $I_1/I = \langle \mathbf{X}^\mathbf{e}
%\rangle$ has dimension $1$ as $\Bbbk-$vector space by Proposition
%\ref{Prop 9.5KM}. 
\end{proof}

\begin{example} 
If $I_1 = \langle X - Y, Y^2 \rangle \subset
  \Bbbk[X,Y]$, then $I_0 = I_1 \cap \langle X-1,Y-1 \rangle = \langle
  X-Y,Y^3-Y^2 \rangle$. This can be verified as follows. 
\begin{verbatim}
      R = QQ[X,Y];
      I1 = ideal(X-Y, Y^2);
      Iaug = ideal(X-1,Y-1);
      I0 = intersect(I1,Iaug);
      mingens I0;
\end{verbatim}
Note that these ideals already appeared in Examples~\ref{ejem10b} (i)
and \ref{ejem10} \ref{item{iv}}.
\end{example}

\begin{remark}
\label{rmk:differentIdealSameCongruence}
The previous results highlight one way in which two different binomial
ideals in $\Bbbk[\mathbf{X}]$ induce the same congruence on
$\mathbb{N}^n$, namely if one contains the other, the congruence has a
nil element, and the larger ideal contains monomials corresponding to
the nil class, while the smaller ideal has no monomials.

There is another way to produce binomial ideals inducing the same
congruence. Let $I$ be a binomial ideal in 
$\Bbbk[\mathbf{X}]$, and let $\mu_1,\dots,\mu_n \in \Bbbk^*$. Consider
the ring isomorphism $\Bbbk[\mathbf{X}] \to \Bbbk[\mathbf{X}]$ given
by $X_i \mapsto \mu_i X_i$ for $i=1,\dots,n$. (This kind of
isomorphism is known as \textbf{rescaling the variables}.)  Then the
image of $I$ is a binomial ideal, which induces the same congruence as
$I$. Indeed, the effect on $I$ of rescaling the variables is to change
the coefficients of the binomials in $I$ by a nonzero multiple, which
does not alter the exponents of those monomials. 

In Theorem~\ref{KMThm9.12}, the ideal $I_0$ can be made
unital by rescaling the variables, by using that $\Bbbk$ is algebraically closed if necessary. The ideal obtained this way
equals the ideal introduced in~\eqref{ecu2}.
\end{remark}

We are now ready to introduce the binomial ideal associated to a
congruence in $\mathbb{N}^n$.

\begin{definition}
\label{def:binomialIdealAssocCong}
Given a congruence $\sim$ on $\mathbb{N}^n$, denote by $I_\sim$ the
unital binomial ideal of $\Bbbk[\mathbf X]$ which is maximal among all proper binomial ideals inducing $\sim$. We say that $I_\sim$ is the \textbf{binomial ideal associated}
to $\sim$.
\end{definition}

To close this section, we introduce one final notion. 

\begin{definition}
\label{def:intersectionOfCongruences}
Let $\sim_1$ and $\sim_2$ be congruences on $\mathbb{N}^n$. The
\textbf{intersection}  $\sim$ of $\sim_1$ and $\sim_2$, denoted $\sim=\sim_1 \cap \sim_2$,
is the congruence on $\mathbb{N}^n$ defined by 
$\mathbf{u} \sim \mathbf{v}$ if and only if $\mathbf{u}
\sim_1 \mathbf{v}$ and $\mathbf{u} \sim_2 \mathbf{v}$.
\end{definition}

From the point of view of equivalence relations, the equivalence
classes of $\sim_1 \cap \sim_2$ form a partition of $\mathbb{N}^n$
which is the common refinement of the partitions induced by $\sim_1$
and $\sim_2$. The following result motivates the use of the
intersection notation and terminology: the intersection of congruences corresponds to the ideal generated by the binomials in the intersection of their associated binomial ideals.

\begin{proposition}\label{Prop:refinement-intersection}
Let $\sim, \sim_1$ and $\sim_2$ be congruences on $\mathbb{N}^n$
whose associated ideals 
in $\Bbbk[\mathbf{X}]$ \text{\rm({Definition~\ref{def:binomialIdealAssocCong}})}
are $I_\sim, I_{\sim_1}$ and $I_{\sim_2},$ respectively. Then $\sim =
\sim_1 \cap \sim_2$ if and only if $I_\sim \subseteq I_{\sim_1} \cap I_{\sim_2}$, and the equality holds if and only if $I_{\sim_1} \cap I_{\sim_2}$ is a binomial ideal.
\end{proposition} 

\begin{proof}
The statement $\mathbf{u} \sim \mathbf{v}$ if and only if $\mathbf{u}
\sim_1 \mathbf{v}$ and $\mathbf{u} \sim_2 \mathbf{v}$ is exactly the
same as $\mathbf{X}^\mathbf{u} - \mathbf{X}^\mathbf{v} \in I_\sim$ if and only if $\mathbf{X}^\mathbf{u} - \mathbf{X}^\mathbf{v} \in I_{\sim_1}$ and $\mathbf{X}^\mathbf{u} - \mathbf{X}^\mathbf{v} \in I_{\sim_2}$. The direct implication of the last statement follows from Theorem \ref{KMThm9.12} and its converse is trivially true because $I_\sim$ is a binomial ideal.
\end{proof}

The following example illustrates the last statement above.

\begin{example}
Let $\sim_1$ and $\sim_2$ be the congruences on $\mathbb{N}^2$ such that $\mathbf{u} \sim_1 \mathbf{v}$ if $\mathbf{u}-\mathbf{v} \in \mathbb{Z}(2,-2)$ and $\mathbf{u} \sim_2 \mathbf{v}$ if $\mathbf{u}-\mathbf{v} \in \mathbb{Z}(3,-3)$, respectively. The binomial ideals of $\mathbb{Q}[X,Y]$ associated to $\sim_1$ and $\sim_2$ are $I_{\sim_1} = \langle X^2 - Y^2 \rangle$ and $I_{\sim_2} = \langle X^3 - Y^3 \rangle$, respectively. Clearly, the binomial ideal associated to $\sim = \sim_1 \cap \sim_2$ is $I_\sim = \langle X^6 - Y^6 \rangle$. Whereas, $I_{\sim_1} \cap I_{\sim_2} = \langle X^4+X^3Y-XY^3-Y^4 \rangle$:
\begin{verbatim}
      R = QQ[X,Y];
      I1 = ideal(X^2-Y^2);
      I2 = ideal(X^3-Y^3);
      intersect(I1,I2);
\end{verbatim}
\end{example}

%%%%%%%%%%%%%%%%%%%%%%%%%%%%%%%%%%%%%%%%%%%%%%%%%%%%%%%%%%%%%%%%%%%%
%%%%%%%%%%%%%%%%%%%%%%%%%%%%%%%%%%%%%%%%%%%%%%%%%%%%%%%%%%%%%%%%%%%%
\section{Toric, lattice and mesoprime ideals}

This section is devoted to the (finitely generated abelian) monoids contained in a group. 

Let $(G,+)$ be a finitely generated abelian group and let $\mathcal{A} = \{\mathbf{a}_1, \ldots, \mathbf{a}_n\}$ be a given subset  of $G$, we
consider the subsemigroup $S$ of $G$ generated by $\mathcal{A},$ that
is to say, 
\[
S = \mathbb{N} \mathbf{a}_1 + \ldots + \mathbb{N}
\mathbf{a}_n.
\]
Since $0 \in \mathbb{N}$, the semigroup $S$ is
actually a monoid.  We may define a surjective monoid map as follows
\begin{equation}
\label{eqn:degA}
\deg_\mathcal{A} : \mathbb{N}^n \longrightarrow S; \quad
\mathbf{u} =(u_1, \ldots, u_n) \longmapsto \deg_\mathcal{A}(\mathbf{u}) =
\sum_{i=1}^n u_i \mathbf{a}_i.
\end{equation}
In the literature, this map is called the factorization
map of $S$ and accordingly, the fiber
$\deg^{-1}_\mathcal{A}(\mathbf{a})$ is called the set of
factorizations of $\mathbf{a} \in S$. 

%Notice that $\deg_\mathcal{A}$ naturally defines as a group homomorphism from $\mathbb{Z}^n$ to $G$; this is the main difference with the monoid morphism considered in \eqref{ecu1}.

Clearly $\deg_{\mathcal{A}}(-)$ determines a congruence on
$\mathbb{N}^n$; in fact, it is the congruence on
$\mathbb{N}^n$ whose presentation map is precisely $\deg_{\mathcal{A}}(-)$ (cf. \eqref{ecu1}). Therefore, if $\widehat{\deg_\mathcal{A}}$ is the map defined in~\eqref{ecu2a}, namely, 
\[\widehat{\deg_\mathcal{A}} : \Bbbk[\mathbb{N}^n] = \Bbbk[\mathbf{X}]  \rightarrow  \Bbbk[S] \ ; \quad
\mathbf{X}^\mathbf{u}  \mapsto  \chi^{\deg_{\mathcal{A}}(\mathbf{u})},  
\] by Theorem~\ref{Th 1.1}, we have that $I_{\mathcal{A}}=\ker(\widehat{\deg_\mathcal{A}})$ is 
spanned as a $\Bbbk-$vector space by
the set of binomials 
\begin{equation}
\label{eqn:binomials}
\{ \mathbf{X}^\mathbf{u} - \mathbf{X}^\mathbf{v} \mid \mathbf{u},
\mathbf{v} \in \mathbb{N}^n\ \text{with}\ \deg_\mathcal{A}(\mathbf{u})
= \deg_\mathcal{A}(\mathbf{v}) \}.
\end{equation}

Observe that $\Bbbk[\mathbf{X}]$ is $S-$graded via $\deg(X_i) =
\mathbf{a}_i,\ i = 1, \ldots, n$. This grading is known as the
\textbf{$\mathcal{A}-$grading} on $\Bbbk[\mathbf{X}]$. 
The semigroup algebra $\Bbbk[S] = \oplus_{\mathbf a \in S} {\rm Span}_\Bbbk \{ \chi^\mathbf{a}\}$ also has a natural $S-$grading. Under these gradings, the map of semigroup algebras $\widehat{\deg_\mathcal{A}}$ is a graded map.
Hence, the ideal $I_\mathcal{A} = \ker(\widehat{\deg_\mathcal{A}})$ is $S-$homogeneous. 

\begin{proposition}\label{Prop rc_ecu1}
Use the notation introduced above, and 
assume that $\mathbf{a}_1,\dots,\mathbf{a}_n$ are nonzero.
The following are equivalent:
\begin{enumerate}
\item \label{item-a}
The fibers of map $\deg_{\mathcal{A}}(-)$ are finite.
\item \label{item-b}
$\deg^{-1}_{\mathcal{A}}(\mathbf{0})= \{ (0, \ldots, 0) \}$.
\item \label{item-c}
$S \cap (-S) = \{0\}$, that is to say,  $\mathbf{a} \in S$ and
  $-\mathbf{a} \in S \Rightarrow \mathbf{a} = \mathbf{0}$. 
\item \label{item-d}
The relation $\mathbf{a}' \preceq \mathbf{a} \Longleftrightarrow
  \mathbf{a}' - \mathbf{a} \in S$ is a partial order on $S$. 
\end{enumerate}
\end{proposition} 

\begin{proof}
Before we proceed with the proof, we note that if one of the
$\mathbf{a}_i$ is zero, then this result is false. For example, let
$G=\mathbb{Z}$, $\mathcal{A}=\{ \mathbf{a}_1=0, \mathbf{a}_2=1\}$. Then
$S=\mathbb{N}$, for which~\ref{item-c} and~\ref{item-d} hold, but
$\deg_{\mathcal{A}}(-)$ does not satisfy either~\ref{item-a} or~\ref{item-b}.

\ref{item-a}$\Rightarrow$\ref{item-b} If $\mathbf{u} \in
\deg^{-1}_{\mathcal{A}}(\mathbf{0})$,
then for every $\ell \in \mathbb{N}$, $\ell \mathbf{u} \in
\deg^{-1}_{\mathcal{A}}(\mathbf{0})$. If $\mathbf{u}\neq (0,\dots,0)$,
then $\deg^{-1}_{\mathcal{A}}(\mathbf{0})$ is infinite.

\ref{item-a}$\Leftarrow$\ref{item-b}
Dickson's Lemma states that any nonempty subset of $\mathbb{N}^n$ has
finitely many minimal elements with respect to the partial order given
by coordinatewise $\leq$. Suppose that
$\deg_{\mathcal{A}}^{-1}(\mathbf{a})$ is infinite. Then by Dickson's
Lemma there exists $\mathbf{u} \in \deg_{\mathcal{A}}^{-1}(\mathbf{a})$
which is not minimal, and therefore there is also $\mathbf{v} \in
\deg_{\mathcal{A}}^{-1}(\mathbf{a})$ such that
$\mathbf{v} \leq \mathbf{u}$ coordinatewise. We conclude that
$\mathbf{u}-\mathbf{v} \in \mathbb{N}^n$ is a nonzero element of
$\deg_{\mathcal{A}}^{-1}(\mathbf{0})$. 

\ref{item-b}$\Rightarrow$\ref{item-c} Let $\mathbf{u}, \mathbf{v} \in
\mathbb{N}^n$ be such that $\deg_{\mathcal{A}}(\mathbf{u})= \mathbf{a}$
and $\deg_{\mathcal{A}}(\mathbf{v})= -\mathbf{a}$. Then 
$\deg_{\mathcal{A}}(\mathbf{u}+\mathbf{v}) = \mathbf{0}$, so that
$\mathbf{u}+\mathbf{v}=(0,\dots,0)$, and therefore
$\mathbf{u}=\mathbf{v}=(0,\dots,0)$, which implies that
$\mathbf{a}=\mathbf{0}$. 

\ref{item-b}$\Leftarrow$\ref{item-c}
Let $\mathbf{u} \in \deg^{-1}_{\mathcal{A}}(\mathbf{0})$. If
$\mathbf{u}_1,\mathbf{u}_2 \in \mathbb{N}^n$ are such that
$\mathbf{u}=\mathbf{u}_1+\mathbf{u}_2$, then by~\ref{item-c}
$\deg_{\mathcal{A}}(\mathbf{u}_1)=\deg_{\mathcal{A}}(\mathbf{u}_2) =
\mathbf{0}$. Repeatedly applying this argument, we conclude that
if $\mathbf{u}\neq 0$, so in particular it has a nonzero coordinate,
then there exists $1\leq i \leq n$ such that
$\mathbf{a}_i=\deg_{\mathcal{A}}(\mathbf{e}_i)=\mathbf{0}$, a contradiction.

\ref{item-c}$\Leftrightarrow$\ref{item-d}
The relation $\preceq$ is always reflexive and transitive.
The fact that $\preceq$ is antisymmetric is equivalent to~\ref{item-c}. 
\end{proof}

%\begin{proposition}
%\end{proposition}

\begin{remark}\label{remark_pos}
If the conditions of Proposition~\ref{Prop rc_ecu1} hold, the monoid
$S$ generated by $\mathcal{A}$ is said to be \textbf{positive}. When
$S$ is positive, $\mathfrak{m} =
\langle X_1, \ldots, X_n \rangle$ is the only 
$S-$homogeneous maximal ideal in $\Bbbk[\mathbf{X}]$. Recall that 
a graded ideal $\mathfrak{m}$ in a graded ring $R$ is a \textbf{graded
  maximal ideal} or \textbf{$^*$maximal ideal} if the only graded
ideal properly containing $\mathfrak{m}$ is $R$ itself.
Graded rings with a unique graded maximal ideal are known as \textbf{graded
local rings} or \textbf{$^*$local rings}. Many results valid for local rings are
also valid for graded local rings, starting with Nakayama's Lemma. In particular, the
%In this sense, $\Bbbk[\mathbf{X}]$ is a local
%ring and a 
\emph{minimal free resolution} of any finitely generated
$\mathcal{A}-$graded $\Bbbk[\mathbf{X}]-$module is well-defined (see
\cite[Section 1.5]{BH} and \cite{Collectanea}). 
\end{remark}

All the monoids in this section are contained in a group. The next result characterize the condition for a monoid to be contained in a group. To state it we need to introduce the following concepts.

\begin{definition}
Let $\sim$ be a congruence on $\mathbb{N}^n$. We will say that $\mathbf{a} \in \mathbb{N}^n/\!\sim$ is \textbf{cancellable} if $\mathbf{b} + \mathbf{a} = \mathbf{c} + \mathbf{a} \Rightarrow \mathbf{b} = \mathbf{c}$, for all $\mathbf{b}, \mathbf{c} \in \mathbb{N}^n/\!\sim$. A monoid is said to be \textbf{cancellative} if all its elements are cancellable. A congruence $\sim$ on $\mathbb{N}^n$ is cancellative if the monoid $\mathbb{N}^n/\!\sim$ is cancellative.
\end{definition}

In the part \ref{item{iv}} of Example \ref{ejem10}, an example of non-cancellative monoid is exhibited.

\begin{proposition}\label{prop:cancellative}
A (finitely generated commutative) monoid is contained in a group if and only if it is cancellative.
\end{proposition}

\begin{proof}
The direct implication is clear. Conversely, if $S = \mathbb{N}^n/\sim$ is a cancellative finitely generated commutative monoid, then $\sim$ can be extended on $\mathbb{Z}^n$ as follows: $\mathbf{u} \sim \mathbf{v}$ if $\mathbf{u} + \mathbf{e} \sim \mathbf{v} + \mathbf{e}$ for some (any) $\mathbf{e} \in \mathbb{N}^n$ such that $\mathbf{u} + \mathbf{e}$ and $\mathbf{v} + \mathbf{e} \in \mathbb{N}^n$. Since $G = \mathbb{Z}^n/\sim$ has a natural group structure and $S \subseteq G$, we are done. 
\end{proof}

The above result shows that our definition of cancellative congruence is equivalent to the usual one (see \cite[p.~44]{Gilmer}).

%%%%%%%%%%%%%%%%%%%%%%%%%%%%%%%%%%%%%%%%%%%%%%%%%%%%%%%%%%%%%%%%%%%%
\subsection{Toric ideals and toric congruences}\mbox{}\par

%\medskip
Suppose now that $G$ is torsion-free and let $G(\mathcal{A})$ denote the subgroup of $G$ generated by $\mathcal{A}$. Since $G$ is torsion-free, then $G \cong \mathbb{Z}^m$, for some $m$. Thus, the semigroup $S$ is isomorphic to a subsemigroup of $\mathbb{Z}^m$. In this case, $S$ is said to be an \textbf{affine semigroup} and  the ideal $I_\mathcal{A}$ is called the \textbf{toric ideal} associated to $\mathcal{A}$. 

Without loss of generality, we may assume that $\mathbf{a}_i \in \mathbb{Z}^d,$ for every $i = 1, \ldots, n$, with $d = \mathrm{rank}(G(\mathcal{A})) \leq m$. Moreover, one can prove that, if $\mathcal{A}$ generates a positive monoid  (see Remark \ref{remark_pos}), there exists a monoid isomorphism under which $\mathbf a_i$ is mapped to an element of $\mathbb{N}^n,\ i = 1, \ldots, n$ (see, e.g.  \cite[Proposition 6.1.5]{BH}), which justifies the use of the term ``positive''.

\begin{lemma}\label{Prop Toric}
If $\mathcal{A} = \{\mathbf{a}_1, \ldots, \mathbf{a}_n\} \subset \mathbb{Z}^d$, then $I_\mathcal{A}$ is prime. 
\end{lemma}

\begin{proof}
By hypothesis, we have that $\Bbbk[S]$ is isomorphic to the subring $
\Bbbk[\mathbf{t}^{\mathbf{a}_1}, \ldots, \mathbf{t}^{\mathbf{a}_n}]$ of the Laurent polynomial ring $\Bbbk[\mathbb{Z}^d] = \Bbbk[t_1^\pm, \ldots, t_d^\pm]$; in particular, $\Bbbk[S] \cong \Bbbk[\mathbf{X}]/I_\mathcal{A}$ is a domain. Therefore $I_\mathcal{A}$ is prime. 
\end{proof}

\begin{theorem}\label{Th CarPrimos}
Let $I$ be a binomial ideal of $\Bbbk[\mathbf{X}]$. The ideal $I$ is prime if and only if there exists $\mathcal{A} = \{\mathbf{a}_1, \ldots, \mathbf{a}_r\}\subset \mathbb{Z}^d$ such that $$I = I_\mathcal{A}\, \Bbbk[\mathbf{X}] + \langle X_{r+1}, \ldots, X_n \rangle,$$ up to permutation and rescaling of variables.
\end{theorem}

\begin{proof}
Suppose that $I$ is prime. If $I$ contains monomials, there exists a set of variables, say $X_{r+1}, \ldots, X_n$ (by permuting variables if necessary), such that $I$ is equal to $I'\,  \Bbbk[\mathbf{X}] +  \langle X_{r+1}, \ldots, X_n \rangle$ where $I'$ is a pure prime binomial ideal of $\Bbbk[X_1, \ldots, X_r]$. Therefore, without loss of generality, we may suppose $I=I'$ and $r=n$. Now, by Theorem \ref{Prop1.11ES}, $\Bbbk[\mathbf{X}]/I \cong \Bbbk[S] = \bigoplus_{\mathbf{a} \in S} {\rm Span}_\Bbbk \{ \chi^\mathbf{a} \}$, for some commutative monoid $S$ generated by $\mathcal{A} = \{\mathbf{a}_1, \ldots, \mathbf{a}_n\}$. Recall that the above isomorphism maps $X_i$ to $\lambda_i \chi^{\mathbf{a}_i}$ for some $\lambda_i \in \Bbbk^*,\ i = 1, \ldots, n$. So, by rescaling variables if necessary, we may assume $\lambda_i = 1$ for every $i$. Now, if $S$ is not contained in a group, by Proposition \ref{prop:cancellative}, there exist $\mathbf{a}, \mathbf{a}'$ and $\mathbf{b} \in S$ such that $\mathbf{a} + \mathbf{b} = \mathbf{a}' + \mathbf{b}$ and $\mathbf{a} \neq \mathbf{a}'$. Thus, $\mathbf{X}^{\mathbf{v}} ( \mathbf{X}^\mathbf{u} - \mathbf{X}^{\mathbf{u}'}) \in I$, but $ \mathbf{X}^\mathbf{u} - \mathbf{X}^{\mathbf{u}'} \not\in I$, where $\mathbf{u} \in \deg_\mathcal{A}^{-1}(\mathbf{a}),$ $\mathbf{u}' \in \deg_\mathcal{A}^{-1}(\mathbf{a}')$ and $\mathbf{v} \in \deg_\mathcal{A}^{-1}(\mathbf{b})$. So, since $I$ is prime, we have that $\mathbf{X}^{\mathbf{v}} \in I$ which is a contradiction. On other hand, if $G(S)$ has torsion, there exist two different elements $\mathbf{a}$ and $\mathbf{a}' \in S$ such that $n \mathbf{a} = n \mathbf{a}'$ for some $n \in 	\mathbb{N}$.  Therefore $ \mathbf{X}^{n \mathbf{u}} - \mathbf{X}^{n \mathbf{u}'} \in I$, where $\mathbf{u} \in \deg_\mathcal{A}^{-1}(\mathbf{a})$ and $\mathbf{u}' \in \deg_\mathcal{A}^{-1}(\mathbf{a}')$. Since $\Bbbk$ is algebraically closed and $I$ is prime, $\mathbf{X}^{\mathbf{u}} - \zeta_n \mathbf{X}^{\mathbf{u}'} \in I$, where $\zeta_n$ is a $n-$th root of unity; in particular,  $\mathbf{a} = \mathbf{a'}$ which is a contradiction. Putting all this together, we conclude that $S$ is an affine semigroup.

The opposite implication is a direct consequence of Lemma \ref{Prop Toric}.
\end{proof}

\begin{definition}
A congruence $\sim$ on $\mathbb{N}^n$ is said to be \textbf{toric} if the ideal $I_\sim$ is prime.
\end{definition}

The following result proves that our definition agrees with the one given in \cite{KM}.

\begin{corollary}
A congruence $\sim$ on $\mathbb{N}^n$ is toric if and only if the non-nil elements of $\mathbb{N}^n/\!\sim$ form an affine semigroup.
\end{corollary}

\begin{proof}
The direct implication follows from Theorem \ref{Th CarPrimos}. Conversely, we assume that the non-nil elements of $\mathbb{N}^n/\!\sim$ form an affine semigroup $S.$ In this case, we have that $[\mathbf u] + [\mathbf v] = \infty$ implies $[\mathbf u] = \infty$ or $[\mathbf v] = \infty$, for every $\mathbf{u}$ and $\mathbf{v} \in \mathbb{N}^n,$ because $S$ is contained in a group and groups have no nil element. Therefore, since $\mathbb{N}^n/\!\sim$ is generated by the classes $[\mathbf{e}_i]$ modulo $\sim$, $i = 1, \ldots, n$, we obtain that $S$ is generated by $\mathcal{A}  = \{[\mathbf{e}_i] \neq \infty\ \mid\ i = 1, \ldots, n\}$ Now, applying Theorem \ref{Th CarPrimos} again, we conclude that $I_\sim$ is a prime ideal.
\end{proof}

%\medskip
%%%%%%%%%%%%%%%%%%%%%%%%%%%%%%%%%%%%%%%%%%%%%%%%%%%%%%%%%%%%%%%%%%%%
\subsection{Lattice ideals and cancellative congruences}\mbox{}\par

Consider now a subgroup $\mathcal{L}$ of $\mathbb{Z}^n$ and define the following congruence $\sim$ on $\mathbb{N}^n$: \[\mathbf{u} \sim \mathbf{v} \Longleftrightarrow\mathbf{u} - \mathbf{v} \in \mathcal{L}.\] Clearly, $\mathbb{N}^n/\!\sim$ is contained in the group $\mathbb{Z}^n/\mathcal{L}$ and the associated ideal $I_\sim$ is equal to $$I_\mathcal{L} := \{ \mathbf{X}^\mathbf{u} - \mathbf{X}^\mathbf{v}\ \mid \mathbf{u}-\mathbf{v} \in \mathcal{L}\}.$$ 

The subgroups of $\mathbb{Z}^n$ are also called lattices. This justifies the term ``lattice'' in the following definition.

%\medskip
\begin{definition}
Let $\mathcal{L}$ be a subgroup of $\mathbb{Z}^n$ and $\rho : \mathcal{L} \to \Bbbk^*$ be a group homomorphism. The lattice ideal corresponding to $\mathcal{L}$ and $\rho$ is $$I_\mathcal{L}(\rho) := \langle \mathbf{X}^{\mathbf{u}} - \rho(\mathbf{u}-\mathbf{v}) \mathbf{X}^{\mathbf{v}}\ \mid\ \mathbf{u} - \mathbf{v} \in \mathcal{L} \rangle.$$ An ideal $I$ of $\Bbbk[\mathbf X]$ is called a \textbf{lattice ideal} if there is subgroup $\mathcal{L} \subset \mathbb{Z}^n$ and a group homomorphism $\rho : \mathcal{L} \to \Bbbk^*$ such that $I = I_\mathcal{L}(\rho)$.
\end{definition}

Observe that the ideal $I_\mathcal{L}$ above is a lattice ideal for the group homomorphism $\rho : \mathcal{L} \to \Bbbk^*$ such that $\rho(\mathbf{u}) = 1,$ for every $\mathbf{u} \in \mathcal{L}$. Moreover, given a subgroup $\mathcal{L}$ of $\mathbb{Z}^n$, we have that the congruence on $\mathbb{N}^n$ defined by a lattice ideal $I_\mathcal{L}(\rho)$ is the same as the congruence on $\mathbb{N}^n$ defined by $I_\mathcal{L}$, for every group homomorphism $\rho : \mathcal{L} \to \Bbbk^*$.

Let us characterize the cancellative congruences on $\mathbb{N}^n$ in terms of their associated binomial ideals. In order to do this, we first recall the following result from \cite{ES96}.

\begin{proposition}\label{Prop Cor2.5ES}\cite[Corollary 2.5]{ES96} If $I$ is a pure binomial ideal of $\Bbbk[\mathbf{X}]$, then there is a unique group morphism $\rho : \mathcal{L} \subseteq \mathbb{Z}^n \to \Bbbk^*$ such that $I : (\prod_{i=1}^n X_i)^\infty = I_\mathcal{L}(\rho)$.
\end{proposition}

Observe that from Proposition \ref{Prop Cor2.5ES}, it follows that no monomial is a zero divisor modulo a lattice ideal.

\begin{corollary}\label{Cor Cong:lattice}
A congruence $\sim$ on $\mathbb{N}^n$ is cancellative if and only if $I_\sim$ is a lattice ideal.
\end{corollary}

\begin{proof}
By Proposition \ref{prop:cancellative}, $\sim$ is cancellative if and only if $\mathbb{N}^n/\sim$ is contained in a group $G$. Thus, the natural projection $\pi : \mathbb{N}^n \to \mathbb{N}^n/\sim$ can be extended to a group homorphism $\bar \pi : \mathbb{Z}^n \to G$ whose restriction to $\mathbb{N}^n$ is $\pi$. Since the kernel, $\mathcal{L},$ of $\bar\pi$ is a subgroup of $\mathbb{Z}^n$ that defines the same congruence as $\sim$, we conclude that both ideals $I_\sim$ and $I_\mathcal{L}$ are equal. For the converse, we first note that the congruence on $\mathbb{N}^n$ defined by a lattice ideal $I_\mathcal{L}(\rho)$ is the same as the congruence on $\mathbb{N}^n$ defined by $I_\mathcal{L}$, for every group homomorphism $\rho : \mathcal{L} \subset \mathbb{Z}^n \to \Bbbk^*$ (see the comment after equation \eqref{simI}). Now, it suffices to note that if $I_\sim = I_\mathcal{L}$ for some subgroup $\mathcal{L}$ of $\mathbb{Z}^n$, then $\mathbb{N}^n/\sim$ is contained in $\mathbb{Z}^n/\mathcal{L}$.
\end{proof}

Observe that a lattice ideal $I_\mathcal{L}$ is not prime in general. Indeed, $I = \langle X^2 - Y^2 \rangle$ is a lattice ideal corresponding to the subgroup of $\mathbb{Z}^2$ generated by $(2,-2)$ which is clearly not prime. Let us give a necessary and sufficient condition for a lattice ideal to be prime.

\begin{definition}
Let $\mathcal{L}$ be subgroup of $\mathbb{Z}^n$ and set $$\mathrm{Sat}(\mathcal{L}) := (\mathbb{Q} \otimes_\mathbb{Z} \mathcal{L}) \cap \mathbb{Z}^n = \{ \mathbf{u} \in \mathbb{Z}^n\ \mid\ d\, \mathbf{u} \in \mathcal{L}\ \text{for some}\ d \in \mathbb{Z} \}.$$ Clearly, $\mathrm{Sat}(\mathcal{L})$ is subgroup of $\mathbb{Z}^n$ and it is called the \textbf{saturation} of $\mathcal{L}$. We say that $\mathcal{L}$ is \textbf{saturated} if $\mathcal L = \mathrm{Sat}(\mathcal L)$.
\end{definition}

\begin{proposition}
A lattice ideal $I_\mathcal{L}(\rho)$ is prime if and only if $\mathcal{L}$ is saturated.
\end{proposition}

\begin{proof}
By using the same argument as in the proof of Corollary \ref{Cor Cong:lattice}, we obtain that $\mathbb{Z}^n/\!\sim\, = \mathbb{Z}^n/\mathcal{L}$, where $\sim$ the congruence defined by $I_\mathcal{L}(\rho)$ on $\mathbb{Z}^n$. Now, since $\mathbb{Z}^n/\mathcal{L}$ is the group generated by $\mathbb{N}^n/\!\sim$, and $\mathbb{Z}^n/\mathcal{L}$ is torsion-free if and only if $\mathcal{L}$ is saturated, we obtain the desired equivalence.
\end{proof}

Notice that the congruence defined by $\mathcal{L}$ is contained in the congruence defined by $\mathrm{Sat}(\mathcal{L})$. In fact,  $\mathrm{Sat}(\mathcal{L})$ defines the smallest toric congruence on $\mathbb{N}^n$ containing the congruence defined by $\mathcal{L}$ on $\mathbb{N}^n$. Therefore, we may say each cancellative congruence \emph{has exactly one toric congruence associated}.

The primary decomposition of a lattice ideal $I_\mathcal{L}(\rho)$ can be completely described in terms of $\mathcal{L}$ and $\rho$. Let us reproduce this result. For this purpose, we need additional notation.

\begin{definition}
If $p$ is a prime number, we define $\mathrm{Sat}_p( \mathcal{L})$ and ${\rm Sat'}_p( \mathcal{L})$ to be the largest sublattices of $\mathrm{Sat}( \mathcal{L})$ containing $L$ such that $\mathrm{Sat}_p( \mathcal{L})/ \mathcal{L}$ has order a power of $p$ and ${\rm Sat'}_p( \mathcal{L})/ \mathcal{L}$ has order relatively prime to $p.$ If $p=0,$ we adopt the convention that $\mathrm{Sat}_p( \mathcal{L})= \mathcal{L}$ and ${\rm Sat'}_p( \mathcal{L})=\mathrm{Sat}( \mathcal{L}).$
\end{definition}

\begin{theorem}\cite[Corollaries 2.2 and 2.5]{ES96}\label{Corolario 2.2ES}
Let $\mathrm{char}(\Bbbk)=p \geq 0$ and consider a group morphsim $\rho : \mathcal{L} \subseteq \mathbb{Z}^n \to \Bbbk^*$. If the order of ${\rm Sat'}_p(\mathcal{L})/ \mathcal{L}$ is $g$, there are $g$ distinct group morphisms $\rho_1, \ldots, \rho_g$  extending $\rho$ to $\mathrm{Sat'}_p(\mathcal{L})$ and for each $j \in \{1, \ldots, g\}$ a unique group morphism $\rho'_j$ extending $\rho$ to $\mathrm{Sat}(\mathcal{L})$. Moreover, there is a unique group morphism $\rho'$ extending $\rho$ to $\mathrm{Sat}_p(\mathcal{L})$. The
radical, associated primes and minimal primary decomposition of
$I_\mathcal{L}(\rho) \subset \Bbbk[\mathbf{X}]$ are: $$ \sqrt{I_\mathcal{L}(\rho)}=I_{\mathrm{Sat}_p(\mathcal{L})}(\rho'), $$ $$
{\rm Ass}(\Bbbk[\mathbf X]/I_\mathcal{L}(\rho))= \{ I_{\mathrm{Sat}(\mathcal{L})}(\rho'_j) \mid j=1,\ldots, g \} $$
and $$ I_{\mathcal{L}}(\rho) = \bigcap_{j=1}^g I_{\mathrm{Sat}'_p(\mathcal{L})}(\rho_j) $$ where
$I_{\mathrm{Sat}'_p(\mathcal{L})}(\rho_j)$ is $I_{\mathrm{Sat}(\mathcal{L})}(\rho'_j)$-primary. In particular, if $p=0,$
then $I_\mathcal{L}(\rho)$ is a radical ideal. The associated primes
$I_{\mathrm{Sat}(\mathcal{L})}(\rho'_j)$ of $I_\mathcal{L}(\rho)$ are all minimal and have the same codimension $\mathrm{rank}(\mathcal{L}).$
\end{theorem}

%\medskip
%%%%%%%%%%%%%%%%%%%%%%%%%%%%%%%%%%%%%%%%%%%%%%%%%%%%%%%%%%%%%%%%%%%%
\subsection{Mesoprime ideals and prime congruences}\mbox{}\par

Given $\delta \subseteq \{1, \ldots, n\}$, set  $\mathbb{N}^\delta := \{ (u_1, \ldots, u_n) \in \mathbb{N}^n\ \mid\ u_i = 0,\ \text{for all}\ i \not\in \delta \}$ and define $\mathbb{Z}^\delta$ as the subgroup of $\mathbb{Z}^n$ generated by $\mathbb{N}^\delta$. Morover, if $\delta = \varnothing,$ by convention, then $\mathbb{Z}^\delta = \{ \mathbf{0} \} \subset \mathbb{Z}^n$.

\begin{definition}
Given $\delta \subseteq \{1, \ldots, n\}$ and a group homomorphism $\rho : \mathcal{L} \subseteq \mathbb{Z}^\delta \to \Bbbk^*$, a $\delta-$\textbf{mesoprime ideal} is an ideal of the form $$I_\mathcal{L}(\rho) + \mathfrak{p}_{\delta^c} $$ with $\mathfrak{p}_{\delta^c} := \langle X_j \mid j \not\in \delta \rangle.$  By convention, $\mathfrak{p}_{\varnothing^c} = \langle X_1, \ldots, X_n \rangle$ and $\mathfrak{p}_{\varnothing} = \langle 0 \rangle$.
\end{definition}

\begin{example}
\mbox{}\par
\begin{enumerate}
\item The ideal $\langle X_1^{17} - 1,X_2 \rangle \subset \Bbbk[X_1,X_2]$ is mesoprime for $\delta = \{1\}$ 
\item By Theorem \ref{Th CarPrimos}, every binomial prime ideal is mesoprime, for a suitable $\delta$. 
\item Lattice ideals are mesoprime for $\delta = \{1, \ldots, n\}$.
\end{enumerate}
\end{example}

Due to Theorem \ref{Corolario 2.2ES}, a mesoprime ideal can be understood as a condensed expression that includes all the information necessary to produce the primary decomposition of the ideal simply by using arithmetic arguments.

Observe that the congruence on $\mathbb{N}^n$ defined by $I_\mathcal{L} + \mathfrak{p}_{\delta^c} $ is the same as the congruence defined by $I_\mathcal{L}(\rho) + \mathfrak{p}_{\delta^c} $, for every $\delta \subseteq \{1, \ldots, n\}$ and every group homomorphism $\rho : \mathcal{L} \subseteq \mathbb{Z}^\delta \to \Bbbk^*$.

\begin{lemma}\label{Lema varZD}
Let $\delta \subseteq \{1, \ldots, n\}$. If $I$ is a $\delta-$mesoprime ideal, then $I : X_i = I,$ for all $i \in \delta$. Equivalently,  $I : (\prod_{i \in \delta} X_i)^\infty = I$.
\end{lemma}

\begin{proof}
If $I$ is a $\delta-$mesoprime ideal, there exists $\rho : \mathcal{L} \subseteq \mathbb{Z}^\delta \to \Bbbk^*$ such that $I = I_\mathcal{L}(\rho) + \mathfrak{p}_{\delta^c}.$ Let $X_i f \in I,\ i \in \delta$. We want to show that $f \in I$. So, without loss of generality, we may assume that no term of $f$ lies in $\mathfrak{p}_{\delta^c}$. In this case, $X_i f \in I_\mathcal{L}(\rho)$. Now, by Proposition \ref{Prop Cor2.5ES}, we conclude that $f \in I_\mathcal{L}(\rho)$, and hence $f \in I$.
\end{proof}

\begin{definition}
A congruence $\sim$ on $\mathbb{N}^n$ is said to be \textbf{prime} if the ideal $I_\sim$ is mesoprime for some $\delta \subseteq \{1, \ldots, n\}$.
\end{definition}

Let us prove that this notion of prime congruence is the same as the usual one (see \cite[p.~44]{Gilmer}).

\begin{proposition}
A congruence $\sim$ on $\mathbb{N}^n$ is \textbf{prime} if and only if every element of $\mathbb{N}^n/\!\sim$ is either nil or cancellable.
\end{proposition}

\begin{proof}
If $\sim$ is prime congruence on $\mathbb{N}^n$, then there exist $\delta \subseteq \{1, \ldots, n\}$ and a subgroup $\mathcal{L} \subseteq \mathbb{Z}^\delta$ such that $I_\sim =  I_\mathcal{L} + \mathfrak{p}_{\delta^c}.$ Let $[\mathbf{u}]$ be non-nil and let $ [\mathbf{v}]$ and $[\mathbf{w}] \in \mathbb{N}^n/\!\sim$ be such that $[\mathbf{v}] + [\mathbf{u}] = [\mathbf{v} + \mathbf{u}] = [\mathbf{w} + \mathbf{u}] = [\mathbf{w}] + [\mathbf{u}]$. In particular, $\mathbf{X}^\mathbf{u} (\mathbf{X}^\mathbf{v} - \mathbf{X}^\mathbf{w}) = \mathbf{X}^{\mathbf{v} + \mathbf{u}} - \mathbf{X}^{\mathbf{w} + \mathbf{v}}   \in I_\sim$. Since $[\mathbf{u}]$ is non-nil, $\mathbf{X}^\mathbf{u}$ does not belong to $I_\sim$. Therefore, $\mathbf{u} \in \{X_i\}_{i \in \delta}$ and, by Lemma \ref{Lema varZD}, $\mathbf{X}^\mathbf{v} - \mathbf{X}^\mathbf{w}  \in I_\sim$, that is, $[\mathbf{v}] = [\mathbf{w}]$. So $[\mathbf{u}]$ is cancellable.

Conversely, suppose that every element of $\mathbb{N}^n/\!\sim$ is either nil or cancellable, set $\delta = \{i \in \{1, \ldots, n\} : [\mathbf{e}_i]$ is cancellable$\}$. Clearly, $j \not\in \delta$ if and only if $X_j \in I_\sim$. So, there exist a binomial ideal $J$ in $\Bbbk[\{X_i\}_{i \in \delta}]$ such that $I_\sim = J\,  \Bbbk[\mathbf{X}] + \mathfrak{p}_{\delta^c}$ (if $\delta = \varnothing$, take $J = \langle 0 \rangle$). Moreover, since $[\mathbf e_i]$  is cancellable for every $i \in \delta,$ if $\mathbf{X}^{\mathbf{e}_i} f = X_i f \in J$, for some $i \in \delta$, then $f \in J$. Thus, by Proposition \ref{Prop Cor2.5ES}, $J$ is lattice ideal of $\Bbbk[\{X_i\}_{i \in \delta}]$ and, consequently, $J\, \Bbbk[\mathbf{X}]$ is a lattice ideal. Therefore, $I_\sim = J\,  \Bbbk[\mathbf{X}] + \mathfrak{p}_{\delta^c}$ is a $\delta-$mesoprime ideal and we are done.
\end{proof}

%%%%%%%%%%%%%%%%%%%%%%%%%%%%%%%%%%%%%%%%%%%%%%%%%%%%%%%%%%%%%%%%%%%%
%%%%%%%%%%%%%%%%%%%%%%%%%%%%%%%%%%%%%%%%%%%%%%%%%%%%%%%%%%%%%%%%%%%%
\section{Cellular binomial ideals}

In this section we study the so-called cellular binomial ideals defined by D.~Eisenbud and B.~Sturmfels in \cite{ES96}. Cellular binomial ideals play a central role in the theory of primary decomposition of binomials ideals (see \cite{ES96} and also \cite{EM, OjPie, Oj2}). As in the previous section, we will determine the congruences on $\mathbb{N}^n$ corresponding to those ideals. We will also outline an algorithm to compute a decomposition of a binomial ideal into cellular binomial ideals which will produce (primary) decompositions of the corresponding congruences.

Let us start by defining the notion of cellular ideal.

\begin{definition}\label{Def Cellular}
A proper ideal $I$ of $\Bbbk[\mathbf{X}]$ is \textbf{cellular} if, for some $\delta \subseteq \{1, \ldots, n\},$ we have that
\begin{enumerate}
\item $I : (\prod_{i \in \delta} X_i)^\infty = I$; equivalently $I : X_i = I,$ for every $i \in \delta$,
\item there exists $d_i \in \mathbb{N}$ such that $X_i^{d_i} \in I,$ for every $i \not\in \delta$.
\end{enumerate}
In this case, we say that $I$ is cellular with respect to $\delta$ or, simply, $\delta-$cellular. By convention, the $\varnothing-$cellular ideals are the binomial ideals whose radical is $\langle X_1, \ldots, X_n \rangle$.
\end{definition}

Observe that an ideal $I$ of $\Bbbk[\mathbf{X}]$ is cellular if, and only if, every variable of $\Bbbk[\mathbf{X}]$ is either a nonzerodivisor or nilpotent modulo $I.$ In particular, prime, lattice, mesoprime and primary ideals are cellular.

The following proposition establishes the relationship between cellular binomial and mesoprime ideals.

\begin{proposition}\label{Prop Celular-CP}
Let $\delta \subseteq \{1, \ldots, n\}$. If $I$ is a $\delta-$cellular binomial ideal in $\Bbbk[\mathbf{X}],$ there exists a group morphism $\rho : \mathcal{L} \subseteq \mathbb{Z}^\delta \to \Bbbk^*$ such that
\begin{enumerate}
\item $(I \cap \Bbbk[\{X_i\}_{i \in \delta}])\, \Bbbk[\mathbf X] = I_\mathcal{L}(\rho)$.
\item $I + \mathfrak{p}_{\delta^c} = I_\mathcal{L}(\rho) + \mathfrak{p}_{\delta^ c}$.
\item $\sqrt{I + \mathfrak{p}_{\delta^c}} = \sqrt{I_\mathcal{L}(\rho)} + \mathfrak{p}_{\delta^c}$.
\item $\sqrt{I} = \sqrt{I_\mathcal{L}(\rho)} + \mathfrak{p}_{\delta^c}$.
\end{enumerate}
In particular, the radical of a cellular binomial ideal is a mesoprime ideal, and the minimal associated primes of $I$ are binomial.
\end{proposition}

\begin{proof}
If $\delta = \varnothing$, then $I \cap \Bbbk[\{X_i\}_{i \in \delta}]) = 0$ and it suffices to take $\rho : \{\mathbf{0}\} \to \Bbbk^*; \mathbf{0} \mapsto 1$. So, assume 
without loss of generality that $\delta \neq \varnothing$.
 
In order to prove part (a), we first note that $J := I \cap \Bbbk[\{X_i\}_{i \in \delta}]$ is binomial by Corollary \ref{coro:elimIsBinomial}, and that $J : (\prod_{i \in \delta} X_i)^\infty = J$ by the definition of cellular ideal. Thus, $J\, \Bbbk[\mathbf X] : (\prod_{i = 1}^n X_i)^\infty = J\, \Bbbk[\mathbf X]$ and, by Proposition \ref{Prop Cor2.5ES}, there is a unique group morphism $\rho : \mathcal{L} \subseteq \mathbb{Z}^\delta \to \Bbbk^*$ such that $J\, \Bbbk[\mathbf X] = I_\mathcal{L}(\rho)$. 

Part (b) is an immediate consequence of (a). 

By part (b) and according to the properties of the radical, we have that  $$ \sqrt{I + \mathfrak{p}_{\delta^c}} = \sqrt{I_\mathcal{L}(\rho) + \mathfrak{p}_{\delta^c}} = \sqrt{\sqrt{I_\mathcal{L}(\rho)} + \mathfrak{p}_{\delta^c}}  \supseteq \sqrt{I_\mathcal{L}(\rho)} + \mathfrak{p}_{\delta^c} \supseteq \sqrt{I_\mathcal{L}(\rho)} + \mathfrak{p}_{\delta^c}. $$ On other hand, given $f \in \sqrt{I_\mathcal{L}(\rho)} + \mathfrak{p}_{\delta^c},$ we can write $f =h + \sum_{i \not\in \delta} g_iX_i$ where $h^e \in I_\mathcal{L}(\rho)$ for some $e > 0$. Now, since $I_\mathcal{L}(\rho) \subseteq I_\mathcal{L}(\rho)  +\mathfrak{p}_{\delta^c} = I + \mathfrak{p}_{\delta^c}$, we have that $f^e = \big(h + \sum_{i \not\in \delta} g_iX_i\big)^e \in  I + \mathfrak{p}_{\delta^c}$, that is to say, $f \in \sqrt{I + \mathfrak{p}_{\delta^c}}$. Thus, we obtain that $\sqrt{I + \mathfrak{p}_{\delta^c}} = \sqrt{I_\mathcal{L}(\rho)} + \mathfrak{p}_{\delta^c}$, as claimed in (c). 

For part (d), we observe that $$\sqrt{I_\mathcal{L}(\rho) + \mathfrak{p}_{\delta^c}} = \sqrt{I_\mathcal{L}(\rho) + \langle X_i^{d_i}\ \mid\ i \not\in \delta \rangle},$$ and that $$I_\mathcal{L}(\rho) +  \langle X_i^{d_i}\ \mid\ i \not\in \delta \rangle = (I \cap \Bbbk[\{X_i\}_{i \in \delta}])\, \Bbbk[\mathbf X] + \langle X_i^{d_i}\ \mid\ i \not\in \delta \rangle \subseteq I \subseteq I + \mathfrak{p}_{\delta^c} ,$$ for every $d_i \geq 1,\ i \not\in \delta$. Therefore, taking radicals, by part (c) we conclude that $\sqrt{I_\mathcal{L}(\rho) + \mathfrak{p}_{\delta^c}} = \sqrt{I}=\sqrt{I + \mathfrak{p}_{\delta^c} }$. 

Now, the last statements are direct consequences of the definition of mesoprimary ideal and Theorem \ref{Corolario 2.2ES}.
\end{proof}

In the following definition we introduce the concept of primary congruence on $\mathbb{N}^n$. We prove that our notion of 
primary congruence is equivalent to the one given in \cite[p. 44]{Gilmer}.

\begin{definition}
A congruence $\sim$ on $\mathbb{N}^n$ is said to be \textbf{primary} if the ideal $I_\sim$ is cellular.
\end{definition}

\begin{definition}
Let $\sim$ be a congruence on $\mathbb{N}^n$. An element $\mathbf{a} \in \mathbb{N}^n/\!\sim$ is said to be \textbf{nilpotent} if $d\,  \mathbf{a}$ is nil, for some $d \in \mathbb{N}$.
\end{definition}

\begin{proposition}
A congruence $\sim$ on $\mathbb{N}^n$ is primary if and only if every element of $\mathbb{N}^n/\!\sim$ is nilpotent or cancellable.
\end{proposition}

\begin{proof}
If $\sim$ is a primary congruence on $\mathbb{N}^n$, the binomial associated ideal $I_\sim$ is $\delta-$cellular for some $\delta \subseteq \{1, \ldots, n\}$. Let $[\mathbf{u}]$ be a non-nilpotent element of $\mathbb{N}^n/\!\sim$. Given $[\mathbf{v}]$ and $[\mathbf{w}] \in \mathbb{N}^n/\!\sim$ such that $[\mathbf{v}] + [\mathbf{u}] = [\mathbf{v} + \mathbf{u}] = [\mathbf{w} + \mathbf{u}] = [\mathbf{w}] + [\mathbf{u}]$, we have that $\mathbf{X}^\mathbf{u} (\mathbf{X}^\mathbf{v} - \mathbf{X}^\mathbf{w}) \in I_\sim$. Since $[\mathbf{u}]$ is not nilpotent, $(\mathbf{X}^\mathbf{u})^d \not\in I_\sim$, for every $d \in \mathbb{N}$. Therefore, no variable $X_i$ with $i \not\in \delta$ divides $\mathbf X^{\mathbf{u}}$ and, by the definition of cellular ideal, we conclude that $I_\sim : \mathbf{X}^{\mathbf{u}} = I_\sim$; in particular, $\mathbf{X}^\mathbf{v} - \mathbf{X}^\mathbf{w} \in I_\sim$, that is, $[\mathbf{v}] = [\mathbf{w}]$, and hence $[\mathbf{u}]$ is cancellable. 

Conversely, suppose that every element of $\mathbb{N}^n/\!\sim$ is nilpotent or cancellable. Set $\delta = \{i \in \{1, \ldots, n\} : [\mathbf{e}_i]$ is cancellable$\}$. Clearly, $j \in \delta$ if and only if $X_j$ is a nonzerodivisor modulo $I_\sim$ and $j \not\in \delta$ if and only if $X_j^{d_j} \in I_\sim$, for some $d_j \geq 1$. Therefore, $I_\sim$ is a $\delta-$cellular ideal (see the paragraph just after Definition \ref{Def Cellular}).
\end{proof}

As a consequence, if $\sim$ is a primary congruence on $\mathbb{N}^n$, then, by Proposition \ref{Prop Celular-CP}, $J := \sqrt{I_\sim}$ is a mesoprime ideal. Therefore, associated to $\sim$ there is one and only one prime congruence, $\sim_J$, obtained by removing nilpotent elements.

%\medskip
%%%%%%%%%%%%%%%%%%%%%%%%%%%%%%%%%%%%%%%%%%%%%%%%%%%%%%%%%%%%%%%%%%%%
\subsection{Cellular Decomposition of Binomial Ideals}\mbox{}\par

\begin{definition}\label{Def CD}
A cellular decomposition of an ideal $I \subseteq \Bbbk[\mathbf{X}]$ is an expression of $I$ as an intersection of cellular ideals with respect to different $\delta \subseteq \{1, \ldots, n\},$ say
\begin{equation}\label{ecu cell0} I = \bigcap_{\delta \in \Delta} \mathcal{C}_\delta,\end{equation} for some subset $\Delta$ of the power set of $\{1, \ldots, n\}.$ Moreover, the cellular decomposition (\ref{ecu cell0}) is said to be minimal if  $\mathcal{C}_\delta' \not\supseteq \bigcap_{\delta \in \Delta \setminus \{\delta'\}} \mathcal{C}_\delta$ for every $\delta' \in \Delta$; in this case, the cellular component $\mathcal C_\delta$ is said to be a $\delta-$cellular component of $I.$
\end{definition}

\begin{example}
Every minimal primary decomposition of a monomial ideal $I \subseteq \Bbbk[\mathbf{X}]$ into monomial ideals is a minimal cellular decomposition of $I.$ Consequently, there is non-uniqueness for cellular decomposition in general: consider for instance the following cellular (primary) decomposition $$\langle X^2,XY \rangle = \langle X \rangle \cap \langle X^2, XY, Y^n \rangle,$$ where $n$ can take any positive integral value.
\end{example}

%From the proof of Theorem 7.1 in \cite{ES96} can be deduced an algorithmic procedure which allows us to write a binomial ideal as intersection of cellular binomial ideals. Before presenting this algorithm it is necessary to recall the following elementary result.

Cellular decompositions of an ideal $I$ of $\Bbbk[\mathbf{X}]$ always exist. A simple algorithm for cellular decomposition of binomial ideals can be found in \cite[Algorithm 2]{OjPie}, this algorithm forms part of the \texttt{binomials} package developed by T.~Kahle and it is briefly described below. The interested reader may consult \cite{Kahle} and \cite{OjPie} for further details.

The following result is the key for producing cellular decompositions of binomial ideals into binomial ideals.

\begin{lemma}\label{Lema DesCel}
Let $I$ be a proper binomial ideal in $\Bbbk[\mathbf X].$ If $I$ is not cellular then there exists $i \in \{1, \ldots, n\}$ and a positive integer $d$ such that $I=(I: X_i^d) \cap (I+\langle X_i^d \rangle),$ with $I : X_i^d$ and $I+ \langle X_i^d \rangle$ binomial ideals strictly containing $I.$
\end{lemma}

\begin{proof} 
If $I$ is not cellular, there exists at least one variable $X_i$ which is zerodivisor and not nilpotent modulo $I$. Then, by the Noetherian property of $\Bbbk[\mathbf X]$, there is a positive integer $d$ such that $I : \langle X_i^d \rangle = I : \langle X_i^e \rangle$ for every $e \geq d$. We claim that $I$ decomposes as $(I: X_i^d) \cap (I+\langle X_i^d \rangle)$. Indeed, let $f \in (I: X_i^d) \cap (I+\langle X_i^d \rangle)$ and let $f = g + h X_i^d$ for some $g \in I$. Then $X_i^d f = X_i^d g + h X_i^{2d}$ and, thus $ h X_i^{2 d} = X_i f - X_i g \in I$. That is, $h \in I : \langle X_i^{2d} \rangle = I : \langle X_i^{d} \rangle$. Hence, $h X_i^d \in I$ and, consequently, $f \in I$.  

It remains to see that both $I: X_i^d$ and $I+\langle X_i^d \rangle$ are binomial ideals which strictly contain $I$. On the one hand, the ideal $I+\langle X_i^d \rangle$ is binomial and $I$ is strictly contained in it, as $X_i$ is not nilpotent modulo
$I.$ On the other hand, $I : X_i^d$ is binomial by Corollary \ref{coro:colonIdealIsBinomial}, and $I$ is strictly contained in $I: X_i^d$ because $X_i$ is a zerodivisor modulo $I.$
\end{proof}

Now, by Lemma \ref{Lema DesCel}, if $I$ is not a cellular ideal then we can find two new proper ideals
strictly containing $I.$ If these ideals are cellular then we are
done. Otherwise, we can repeat the same argument with these new
ideals, getting strictly increasing chains of binomial ideals.
Since $\Bbbk[\mathbf X]$ is a Noetherian ring, each one of these chains has to be 
stationary. So, in the end, we obtain a (redundant) cellular
decomposition of $I.$ Observe that this process does not depend on the base field.

\begin{example}\label{Ejem intro}
Consider the binomial ideal $I = \langle X^4 Y^2-Z^6,X^3 Y^2-Z^5,X^2-Y Z \rangle$ of
$\mathbb{Q}[X,Y,Z].$ By using \cite[Algorithm 2]{OjPie} we obtain  the
following cellular decomposition, $I=I_1 \cap I_2 \cap I_3,$ where
$$
\begin{array}{rcl}
I_1 & = & \langle Y-Z,X-Z \rangle \\ I_2 & = & \langle Z^2,X Z,X^2-Y Z \rangle \\ I_3 & = &
\langle X^2-Y Z, X Y^3 Z - Z^5, X Z^5 - Z^6, Z^7, Y^7 \rangle.
\end{array}$$
\begin{verbatim}
      loadPackage "Binomials";      
      R = QQ[X,Y,Z];
      I = ideal(X^4*Y^2-Z^6,X^3*Y^2-Z^5,X^2-Y*Z);
      binomialCellularDecomposition I
\end{verbatim}
\end{example}

As a final conclusion we may notice the following:

\begin{corollary}
Let $\sim$ be a congruence on $\mathbb{N}^n$. A primary decomposition of $\sim$ can be obtained by computing a cellular decomposition of $I_\sim$.
\end{corollary}

\begin{proof}
It is a direct consequence of Proposition \ref{Prop:refinement-intersection} by the definition of primary congruence.
\end{proof}

%%%%%%%%%%%%%%%%%%%%%%%%%%%%%%%%%%%%%%%%%%%%%%%%%%%%%%%%%%%%%%%%%%%%
%%%%%%%%%%%%%%%%%%%%%%%%%%%%%%%%%%%%%%%%%%%%%%%%%%%%%%%%%%%%%%%%%%%%
\section{Mesoprimary ideals}

The main objective of this section is to analyze the mesoprimary ideals and their corresponding congruence. Mesoprimary ideals were introduced by Thomas Kahle and Ezra Miller in \cite{KM} as an intermediate construction between cellular and primary binomial ideals. Kahle and Miller proved combinatorially that every cellular binomial can be decomposed into finitely many mesoprimary ideals over an arbitrary field. However, not every decomposition of a binomial ideal as an intersection of mesoprimary ideals is a mesoprimary decomposition in the sense of Kahle and Miller. These mesoprimary decompositions feature refined combinatorial requirements, and currently there is no algorithm available to compute them. On the other hand, decompositions of binomial ideals into mesoprimary ideals can be produced algorithmically.. Despite of this, mesoprimary decompositions have been successfully used to solve open problems	(see \cite{KMO} and \cite{MO}). 

The following preparatory result will be helpful in understanding what mesoprimary ideals are.

\begin{proposition}\label{Prop Meso1}
Let $I$ be a $\delta-$cellular binomial ideal in $\Bbbk[\mathbf{X}]$. If $\mathbf{X}^\mathbf{u} \in \Bbbk[\{X_i\}_{i\not\in \delta}] \setminus I$, then $I:\mathbf{X}^\mathbf{u}$ is a $\delta-$cellular binomial ideal.
\end{proposition}

\begin{proof}
First of all, we note that  $I :  \mathbf{X}^\mathbf{u} \neq \langle 1 \rangle$ because $\mathbf{X}^\mathbf{u} \not\in I$. Moreover, we have that $I :  \mathbf{X}^\mathbf{u}$ is binomial by Corollary \ref{coro:colonIdealIsBinomial}. Now, since $I : (\prod_{i \in \delta} X_i)^\infty = I$, then  
$$( I:\mathbf{X}^\mathbf{u}) : (\prod_{i \in \delta} X_i)^\infty =  
( I: (\prod_{i \in \delta} X_i)^\infty) : \mathbf{X}^\mathbf{u} = 
I:\mathbf{X}^\mathbf{u}.$$ And, clearly, for every $i \not\in \delta,\ X_i^{d_i} \in I:\mathbf{X}^\mathbf{u}$ for some $d_i \geq 1$ because $I \subseteq I :  \mathbf{X}^\mathbf{u}$. Putting all this together, we conclude that $I:\mathbf{X}^\mathbf{u}$ is a $\delta-$cellular binomial ideal. 
\end{proof}

If $I$ is a $\delta-$cellular binomial ideal, then the ideal $(I:\mathbf{X}^\mathbf{u}) + \mathfrak{p}_{\delta^c}$ is $\delta-$mesoprime by Propositions \ref{Prop Meso1} and \ref{Prop Celular-CP}(b). Moreover, there exists $d_i \geq 1$ such that $X_i^{d_i} \in I$ for each $i \not\in \delta$. Thus there are finitely many mesoprime ideals of the form $(I:\mathbf{X}^\mathbf{u}) + \mathfrak{p}_{\delta^c}$. These are the so-called mesoprimes associated to $I$:

\begin{definition}
Let $I$ be a $\delta-$cellular binomial ideal in $\Bbbk[\mathbf{X}]$. We will say that $I_\mathcal{L}(\rho) + \mathfrak{p}_{\delta^c}$ is a mesoprime ideal associated to $I$ if there exist a monomial $\mathbf{X}^\mathbf{u} \in \Bbbk[\{X_i\}_{i \not\in \delta}]$ such that $$\big((I:\mathbf{X}^\mathbf{u}) \cap \Bbbk[\{X_i\}_{i \in \delta}]\big)\, \Bbbk[\mathbf X] = I_\mathcal{L}(\rho).$$ 
\end{definition}

Now we may introduce the notion of mesoprimary ideal.

\begin{definition}
A binomial ideal is said to be \textbf{mesoprimary} if it is cellular and it has only one associated mesoprime ideal. A congruence $\sim$ on $\mathbb{N}^n$ is mesoprimary if $I_\sim$ is a mesoprimary ideal of $\Bbbk[\mathbf{X}]$
\end{definition}

The following lemma clarifies the notion of mesoprimary ideal.

\begin{lemma}\label{Lema mesoprimary}
A $\delta-$cellular binomial ideal $I$ in $\Bbbk[\mathbf{X}]$ is mesoprimary if and only if $(I:\mathbf{X}^\mathbf{u}) \cap \Bbbk[\{X_i\}_{i \not\in \delta}] = I \cap \Bbbk[\{X_i\}_{i \not\in \delta}] $, for all $\mathbf{X}^\mathbf{u} \in \Bbbk[\{X_i\}_{i\not\in \delta}] \setminus I$.
\end{lemma}

\begin{proof}
It suffices to note that $I$ has two different associated mesoprimes if and only if there exists $\mathbf{X}^\mathbf{u} \in \Bbbk[\{X_i\}_{i\not\in \delta}]$ such that $(I:\mathbf{X}^\mathbf{u}) \cap \Bbbk[\{X_i\}_{i \not\in \delta}] \neq I \cap \Bbbk[\{X_i\}_{i \not\in \delta}] $ because, in this case, by Proposition \ref{Prop Celular-CP}, $(I:\mathbf{X}^\mathbf{u}) + \mathfrak{p}_{\delta^c}$ and $ I  + \mathfrak{p}_{\delta^c}$ are two different associated mesoprimes to $I$.
\end{proof}

\begin{definition}
Let $\sim$ be a congruence on $\mathbb{N}^n$. An element $\mathbf{a} \in \mathbb{N}^n/\!\sim$ is said to be \textbf{partly cancellable} if $\mathbf{a}  + \mathbf{b} = \mathbf{a}  + \mathbf{c} \neq \infty \Rightarrow \mathbf{b} = \mathbf{c},$ for all cancellable $\mathbf{b}, \mathbf{c} \in \mathbb{N}^n$
\end{definition} 

\begin{proposition}
A congruence $\sim$ on $\mathbb{N}^n$ is mesoprimary if and only if it is primary and every element in $\mathbb{N}^n/\!\sim$ is partly cancellable.
\end{proposition}

\begin{proof}
If $\sim$ is a mesoprimary congruence on $\mathbb{N}^n$, then $I= I_\sim$ is $\delta-$cellular for some $\delta \subseteq \{1, \ldots, n\}$. Thus, $\sim$ is primary. Moreover, $(I:\mathbf{X}^\mathbf{u}) \cap \Bbbk[\{X_i\}_{i \in \delta}] = I \cap \Bbbk[\{X_i\}_{i \in \delta}]$, for all $\mathbf{X}^\mathbf{u} \in \Bbbk[\{X_i\}_{i\not\in \delta}] \setminus I$ (equivalently, for all $\mathbf{u} \in \mathbb{N}^n$ such that $[\mathbf{u}]$ is nilpotent and it is not a nil). Therefore, if $[\mathbf{u}] \in \mathbb{N}^n/\!\sim$ is nilpotent and $[\mathbf{v}], [\mathbf{w}]$ are cancellable elements such that $[\mathbf{u}]  + [\mathbf{v}] = [\mathbf{u}] +  [\mathbf{w}] \neq \infty$, then $$\mathbf{X}^\mathbf{v} - \mathbf{X}^\mathbf{w} \in (I:\mathbf{X}^\mathbf{u}) \cap \Bbbk[\{X_i\}_{i \in \delta}] = I \cap \Bbbk[\{X_i\}_{i \in \delta}],$$ that is to say $[\mathbf{v}] = [\mathbf{w}].$ So, $[\mathbf{u}]$ is partly cancellative.

Conversely, suppose that $\sim$ is primary congruence on $\mathbb{N}^n$ such that  every element in $\mathbb{N}^n/\!\sim$ is partly cancellable. Since $\sim$ is primary, we have that $I_\sim$ is $\delta-$cellular, by setting $\delta = \{i \in \{1, \ldots, n\} : [\mathbf{e}_i]$ is cancellable$\}.$
Now, if $\mathbf{X}^\mathbf{u} \in \Bbbk[\{X_i\}_{i\not\in \delta}] \setminus I$, we have that $[\mathbf{u}]$ is partly cancellable. Thus, for every $\mathbf{X}^\mathbf{v} - \mathbf{X}^\mathbf{w} \in \Bbbk[\{X_i\}_{i \in \delta}]$, we have that $\mathbf{X}^\mathbf{u} (\mathbf{X}^\mathbf{v} - \mathbf{X}^\mathbf{w}) \in I \Rightarrow \mathbf{X}^\mathbf{v} - \mathbf{X}^\mathbf{w} \in I$. Therefore, $ (I:\mathbf{X}^\mathbf{u}) \cap \Bbbk[\{X_i\}_{i \in \delta}] \subseteq I \cap \Bbbk[\{X_i\}_{i \in \delta}]$. Now, since the opposite inclusion is always fulfilled, by Lemma \ref{Lema mesoprimary}, we are done.
\end{proof}

There are other intermediate constructions between cellular and primary ideals, such as the unmixed decomposition (see \cite{ES96, OjPie} and, more recently, \cite{EM}). The following example shows that unmixed cellular binomial ideals are not mesoprimary. Recall that an unmixed cellular binomial ideal is a cellular binomial ideal with no embedded associated primes (see \cite[Proposition 2.4]{OjPie}).

\begin{example}
Consider the unmixed cellular binomial $I \subset \Bbbk[X,Y]$ generated by $ \{ X^2-1,Y(X-1),Y^2 \}$ . The ideal $I$ is not mesoprimary, because $$(I : Y) \cap \Bbbk[X] = \langle X - 1 \rangle \neq \langle X^2-1 \rangle = I \cap \Bbbk[X].$$
\begin{verbatim}
      loadPackage "Binomials";
      R = QQ[X,Y]
      I = ideal(X^2-1,Y*(X-1),Y^2)
      cellularBinomialAssociatedPrimes I
      eliminate(I:Y,Y)
      eliminate(I,Y)
\end{verbatim}
\end{example}

We end this section by exhibiting the statement of Kahle and Miller which describes the primary decomposition of a mesoprimary ideal, in order to give an idea of how useful would be to have an algorithm for the mesoprimary decomposition of a cellular binomial ideal.

\begin{proposition}[{\cite[Corollary~15.2 and~Proposition~15.4]{KM}}]
\label{p:mesoprimaryPrimDec}
Let $I$ be a ($\delta$-cellular) mesoprimary ideal, and denote by
$I_{\mathcal{L}}(\rho)$ the lattice ideal $I \cap
\Bbbk[\{X_i\}_{i \in \delta}]$. 
The associated primes of $I$ are exactly the (minimal) primes of its
associated mesoprime $I + 
\mathfrak{p}_{\delta^c}$.
Moreover, if $I_{\mathcal{L}}(\rho) = \cap_{j=1}^g I_j$ is the primary decomposition
of $I_{\mathcal{L}}(\rho)$ from
Theorem~\ref{Prop Cor2.5ES}, then
\[
I = \bigcap_{j=1}^g (I+I_j)
\]
is the primary decomposition of $I$.
\end{proposition}

Notice that the hypothesis $\Bbbk$ algebraically closed is only needed when Theorem~\ref{Prop Cor2.5ES} is applied.

\medskip
\noindent\textbf{Acknowledgement}
We thank the anonymous referees for their detailed suggestions and comments, which have greatly improved this article.
The present paper is based on a course of lectures delivered by the second author at the EACA's Third International School on Computer Algebra and Applications \url{https://www.imus.us.es/EACASCHOOL16/}. He thanks the organizers for giving him that opportunity.

%%%%%%%%%%%%%%%%%%%%%%%%%%%%%%%%%%%%%%%%%%%%%%%%%%%%%%%%%%%%%%%%%%%%

\end{document}